\documentclass[12pt]{amsart}
\usepackage{amssymb}
\newcommand{\slf}{\mathfrak{sl}}
\newcommand{\Part}{\mathcal{P}}
\newcommand{\g}{\mathfrak{g}}
\newcommand{\C}{\mathbb{C}}
\newcommand{\Z}{\mathbb{Z}}

\newcommand{\h}{\mathfrak{h}}
\newcommand{\OCat}{\mathcal{O}}
\newcommand{\WC}{\mathfrak{WC}}
\newcommand{\Sym}{\mathfrak{S}}
\newcommand{\Irr}{\operatorname{Irr}}
\newcommand{\Supp}{\operatorname{Supp}}
\newcommand{\wc}{\mathfrak{wc}}
\newcommand{\param}{\mathfrak{c}}
\newcommand{\Cat}{\mathcal{C}}
\newcommand{\Hom}{\operatorname{Hom}}
\newcommand{\End}{\operatorname{End}}
\newcommand{\Q}{\mathbb{Q}}
\newcommand{\Ext}{\operatorname{Ext}}
\newcommand{\Res}{\operatorname{Res}}
\newcommand{\Ind}{\operatorname{Ind}}
\newcommand{\Tor}{\operatorname{Tor}}
\newcommand{\B}{\mathcal{B}}
\newcommand{\HC}{\operatorname{HC}}
\newcommand{\VA}{\operatorname{V}}
\newcommand{\gr}{\operatorname{gr}}
\newcommand{\cont}{\operatorname{cont}}
\newtheorem{Thm}{Theorem}[section]
\newtheorem{Prop}[Thm]{Proposition}
\newtheorem{Cor}[Thm]{Corollary}
\newtheorem{Lem}[Thm]{Lemma}
\theoremstyle{definition}

\newtheorem{defi}[Thm]{Definition}
\newtheorem{Rem}[Thm]{Remark}

\unitlength=1mm   \numberwithin{equation}{section}
\numberwithin{table}{section} \oddsidemargin=0cm
\title{Supports of simple modules in cyclotomic Cherednik categories $\mathcal{O}$}
\author{Ivan Losev}
\address{Department
of Mathematics, Yale University, New Haven, CT 06511, USA}
\email{ivan.loseu@gmail.com}
\begin{document}
\begin{abstract}
The goal of this paper is to compute the supports of simple modules in
the categories $\mathcal{O}$ for the rational Cherednik algebras associated
to groups $G(\ell,1,n)$. For this we compute some combinatorial maps
on the set of simples: wall-crossing bijections and a certain
$\mathfrak{sl}_\infty$-crystal associated to a Heisenberg algebra
action on a Fock space.
\end{abstract}
\maketitle
\section{Introduction}
We fix  positive integers $\ell,n$ and form the wreath-product
group $W=G(\ell,1,n):=\Sym_n\ltimes (\Z/\ell\Z)^n$. This is a complex
reflection group acting on $\h:=\C^n$. To the pair $(W,\h)$ we can
assign the so called rational Cherednik algebra $H_c$ depending
on a parameter $c$ that is a collection of complex numbers, one
for each conjugacy class of complex reflections in the group $W$.
These algebras were introduced by Etingof and Ginzburg in \cite[Section 4]{EG}.

As a vector space, $H_c=S(\h^*)\otimes \C W\otimes S(\h)$, where
$S(\h^*),\C W, S(\h)$ are subalgebras in $H_c$. The adjoint actions of $W$
on $S(\h),S(\h^*)$ are the usual ones, and there is an interesting commutation
relation  between $y\in \h$ and $x\in \h^*$  depending on the parameter
$c$. We will recall a presentation of $H_c$ by generators and relations
below, Section \ref{SS_RCA}.

One has a distinguished category of $H_c$-modules, the category $\mathcal{O}_c$
introduced in \cite[Section 3]{GGOR} to be recalled in Section \ref{SS_RCA_O}.
This category consists of all modules
that are finitely generated over the subalgebra $S(\h^*)$ and where
$\h$ acts locally nilpotently. Its simple objects are parameterized by
$\Irr(W)$: to $\tau\in \Irr(W)$ we assign the unique simple quotient
$L_c(\tau)$ of the {\it Verma module} $\Delta_c(\tau)=H_c\otimes_{S(\h)\# W}\tau$,
where $\h$ acts on $\tau$ by $0$.

To each module $M\in \mathcal{O}_c$ we can assign its {\it support},
the closed subvariety of $\h$ defined by the annihilator of $M$
in $S(\h^*)=\C[\h]$. The main purpose of this paper is to compute
$\Supp(L_c(\tau))$ combinatorially starting from $\tau$ and $c$.
In particular, this will yield a classification of the finite
dimensional irreducible $H_c$-modules, as those are precisely
the modules whose supports are equal to $\{0\}$.

\subsection{Known results}\label{SS_known_results}
First of all, all possible supports of simples are known, this is
implicit in \cite[Section 3.8]{BE} and explicit in \cite[Section 3.10]{SV}.
Namely, let $\kappa$ be the component of the parameter $c$ corresponding
to the conjugacy class of complex reflections in $W$ intersecting
$\Sym_n$. The case $\kappa=0$ is easy and, in what follows,  we mostly consider $\kappa\neq 0$.
Let $e$ denote the denominator of $\kappa$ presented as an irreducible
fraction if $\kappa$ is rational, we take $e=+\infty$
if $\kappa$ is irrational. 
By \cite{BE}, the support of any simple equals $W \Gamma_{p,q}$, where $p,q$
are non-negative integers satisfying $p+ eq\leqslant n$ (in particular,
if $e>n$, then $q=0$) and $\Gamma_{p,q}$ is the subspace of
$\h$ given by
$$\Gamma_{p,q}=\{(x_1,\ldots,x_p, y_1,\ldots,y_1,\ldots,y_q,\ldots,y_q,0,\ldots,0)\},$$
where we have $q$ groups of $e$ equal elements.
The numbers $p,q$ are determined from $\Gamma_{p,q}$ uniquely whenever $e>1$.
When $e=1$, we always take $p=0$. We write $p(\lambda),q(\lambda)$ (or $p_c(\lambda),
q_c(\lambda)$ if we want to indicate the dependence on the parameter $c$) for the numbers
$p,q$ such that $\Supp(L_c(\lambda))=W\Gamma_{p,q}$.

A combinatorial recipe to compute $p(\lambda)$ was given in \cite{cryst}.
Recall that the irreducible representations of $G(\ell,1,n)$ are
parameterized by the $\ell$-multipartitions $\lambda$ of $n$, denote this set
by $\Part_\ell(n)$. We can encode the remaining $\ell-1$ components
of $c$ as an $\ell$-tuple of complex numbers $s_1,\ldots,s_\ell$
defined up to a common summand. Namely,
there are linear functions $h_i, i=1,\ldots,\ell,$ of the parameter $c$ defined
up to a common summand. We define $s_i$ from $h_i=\kappa s_i-i/\ell$,
see Section \ref{SS_cyclot_O}.

There is a $\hat{\slf}_e^k$-crystal structure
on $\Part_\ell:=\bigsqcup_n \Part_\ell(n)$, where $k$ and the crystal structure
itself depend on $s_1,\ldots,s_\ell$, this will be recalled below in
\ref{SSS_crystal} (when $e=+\infty$, by $\hat{\slf}_\infty$ we mean
$\slf_\infty$). A combinatorial construction of this crystal
was given by Uglov in \cite[Section 2.2]{Uglov}.
According to \cite[Section 5.5]{cryst}, $p(\lambda)$ is the depth of $\lambda$ in the crystal, i.e.,
the minimal number $d$ such that any composition of $d+1$
annihilation crystal operators kills $\lambda$.

Let us explain now what is known about $q(\lambda)$. First, one can reduce the computation
of $q(\lambda)$ to the case when one has $k=1$, see Section \ref{SS_equi} below.

The case $\ell=1$ was done in \cite{Wilcox}. Here we always have
$p+eq=n$. We can divide $\lambda$ by $e$ with remainder:
$\lambda=e\lambda'+\lambda''$, where $\lambda',\lambda''$ are partitions
such that $|\lambda'|$ is maximal possible, the operations are done part-wise. Then we have $q(\lambda)=|\lambda'|$,
see \cite[Theorem 1.6]{Wilcox}.

The number of simples $L(\lambda)$ with given support was computed in
\cite{SV} (under several restrictions on the parameters that
can be removed, as will be explained in \ref{SSS_Heisenberg}).

\subsection{Main results of this paper}\label{SS_main_res}
Based on a construction from \cite{SV},  we will see that, when $k=1$, there is a
level $1$ crystal structure for the algebra
$\slf_\infty$ on $\Part_\ell$ such that each creation operator
adds $e$ boxes to $\lambda$ and $q(\lambda)$ is the depth of $\lambda$
in this crystal, Lemma \ref{Lem:supp_Heis_cryst}.  This crystal is a discrete
analog of the Heisenberg algebra action on a higher level Fock space. The present paper
gives the first construction of this crystal as well as a combinatorial recipe
to compute it. Once the crystal is   determined,
a combinatorial formula for $q(\lambda)$ follows.

One easy thing to observe is that the $\hat{\slf}_e$ and $\slf_\infty$-crystal
commute, see Section \ref{SS_Heis_cryst}.  So it is enough to compute the latter on the singular
(=depth 0) elements for the former.
The computation of the level 1 $\slf_\infty$-crystal consists of
two parts. We first compute it explicitly in an {\it asymptotic chamber}
and then use explicit combinatorial {\it wall-crossing bijections}
to transfer it to the other chambers. 

Let us explain what we mean by chambers. Pick two parameters $c=(\kappa,s_1,\ldots,s_\ell)$
and $c'=(\kappa',s_1',\ldots,s_\ell')$. We say that $c$ and $c'$ have {\it integral difference}
if $\kappa'-\kappa, (h_i-h_j)-(h'_i-h_j')$ are integers for all $i,j$
(recall that the $h_i$'s are defined up to a common summand). The set of parameters
that have integral difference with a given one forms a lattice $\param_{\Z}:=\Z^\ell$ in the space of all
parameters. There are finitely many hyperplanes (depending on $n$)
in $\mathbb{Q}\otimes_{\Z}\param_{\Z}$ that
split $\param_{\Z}$ into the union of cones to be called {\it chambers}. The walls depend on $n$ in such a way
that, as $n$ increases,  we need to add more walls. The categories $\mathcal{O}_c$ are the same
for any $c$ in a given chamber, see \cite[Section 4.2]{rouq_der},
(it is actually enough to assume that $c$ lies in an interior
of  a chamber).

The most interesting case is when $\kappa$ is a rational number with denominator
$e$ and $\kappa e s_1,\ldots,$ $ \kappa e s_\ell$ are all integers. We are going to assume
this until the end of the section. The general case
can be reduced to this one.

There are distinguished chambers to be called {\it asymptotic}. These are $2\ell!$ chambers
containing parameters $c$ such that  $s_1,\ldots,s_\ell$ satisfy $|s_i-s_j|>n$ (the factor of $2$
comes from the choice of a sign of $\kappa$ and $\ell!$ comes from ordering $s_1,\ldots,s_\ell$).
Let us suppose that
\begin{equation}\label{eq:asymp} \kappa<0, s_1\gg s_2\gg\ldots\gg s_\ell.
\end{equation}
Note that the condition that $c$ lies in an asymptotic chamber depends on $n$.
When $n$ becomes large enough, any given parameter stops lying in the asymptotic
chamber.

Then we have  the following result (note that we can impose weaker assumptions on $c$,
see  Proposition \ref{Prop:cryst_part}).

\begin{Prop}\label{Prop:domin_cryst}
Suppose $c$ is as in (\ref{eq:asymp}) and let $\lambda=(\lambda^{(1)},\ldots,
\lambda^{(\ell)})$ be a singular multi-partition for the $\hat{\slf}_e$-crystal.
Then the following holds.
\begin{enumerate}
\item $\lambda^{(\ell)}$ is divisible by $e$, i.e., there is a partition $\lambda'$
such that $\lambda^{(\ell)}=e\lambda'$. We have $q(\lambda)=|\lambda'|$.
\item The annihilation operator $\tilde{e}^\infty_i, i\in \Z,$ for the $\slf_\infty$-crystal
takes $\lambda$ to the multipartition $\underline{\lambda}$ specified by
$\underline{\lambda}^{(j)}=\lambda^{(j)}$ for $j<\ell$ and $\underline{\lambda}^{(\ell)}=e\underline{\lambda}'$,
where $\underline{\lambda}'$ is obtained from $\lambda'$ by removing an $i$-box
(if there are no removable $i$-boxes in $\lambda'$, then we set $\tilde{e}^{\infty}_i\lambda=0$).
\end{enumerate}
\end{Prop}

Now let us explain what happens when we cross a wall. Let $c,c'$ be two parameters
with integral difference lying in two chambers separated by a wall.

\begin{Prop}\label{Prop:wc_bij}
There is a bijection $$\mathfrak{wc}_{c'\leftarrow c}: \bigsqcup_{k\leqslant n} \Part_\ell(m)
\rightarrow \bigsqcup_{k\leqslant n}\Part_\ell(m)$$
that preserves $k$ and  intertwines the annihilation operators for the $\hat{\slf}_e$- and $\slf_\infty$-crystals. In particular, it preserves supports.
This bijection is given by a combinatorial recipe to be explained below in Section \ref{SS_wc_bij_comput}.
\end{Prop}

Let us explain how $\mathfrak{wc}_{c'\leftarrow c}$ arises.
In \cite[Proposition 5.3, Theorem 6.1]{rouq_der} we have established an equivalence $\WC_{c'\leftarrow c}:D^b(\OCat_c)
\xrightarrow{\sim} D^b(\OCat_{c'})$ and have shown that it is perverse (in the sense
of Chuang and Rouquier).
This gives rise to a bijection $\operatorname{Irr}(\OCat_c)\rightarrow \Irr(\OCat_{c'})$
and this is the bijection $\wc_{c'\leftarrow c}$ that we need.
We would like to emphasize that the fact that the bijections preserve
support does not follow from the claim that these bijections come from
perverse equivalences.  The reason why these bijections
preserve supports is that they are given by derived tensor products with Harish-Chandra bimodules,
see Corollary \ref{Cor:wc_bij_supp} for the proof.

We would like to emphasize that we only have a relatively explicit combinatorial recipe for
a bijection between two adjacent chambers\footnote{A more explicit formula for the wall-crossing bijection
was recently obtained by Jacon and
Lecouvey in \cite{JL})}. Of course, for arbitrary chambers we can
take the composition of such bijections. But the resulting bijection is going to
be complicated. Still the Heisenberg crystal operators are computable, e.g., with a computer\footnote{A more explicit combinatorial
formula for the Heisenberg crystal has been recently obtained by Gerber and Norton
in \cite{GN}.}.

\subsection{Other results}
We also compute the wall-crossing bijection for the wall $\kappa=0$
in the case when $\ell=1$, see Section \ref{SS_wc_bij_comput}. This bijection should be thought as an extension
of the Mullineux involution from the modular representation theory of the symmetric groups.
This wall-crossing bijection is expected to play an important role in the study of the representation
theory of type A rational Cherednik algebras over fields of large enough positive characteristic.

\subsection{Content}
Sections \ref{S_RCA} and \ref{S_cyclot} basically do not contain any new results.
Sections \ref{S_wc_further} and \ref{S_main_proofs} are new.

In Section \ref{S_RCA} we recall several known results and constructions for
general rational Cherednik algebras. In Section \ref{SS_RCA} we recall
the definition of these algebras and basic structural results
following \cite{EG}. In Section \ref{SS_hw_cat} we recall
basics about highest weight categories.
In Section \ref{SS_RCA_O} we define categories $\mathcal{O}$.  In Section \ref{SS_HC_bimod}
we recall Harish-Chandra (shortly, HC) bimodules for rational Cherednik algebras.
In Section \ref{SS_compl} we recall isomorphisms of some completions
of rational Cherednik algebras.
Section \ref{SS_IndRes} deals with induction
and restriction functors for categories $\mathcal{O}$, \cite{BE},
and for HC bimodules, \cite{sraco}.
In the next part, Section \ref{SS_chamber}, we elaborate on the chamber
decomposition for the space of parameters. Finally, in Section
\ref{SS_WC} we recall the definition of wall-crossing functors between categories
$\mathcal{O}$ with different parameters introduced in \cite{rouq_der}.
We also recall wall-crossing bijections between the sets of simples
in those categories.

Section \ref{S_wc_further} establishes some new properties of the
wall-crossing functors and the wall-crossing bijections.
First, in Section \ref{SS_wc_bij_gen} we show that a wall-crossing bijection $\wc_{c-\psi\leftarrow c}$
depends not on $c$ but on the wall being crossed.
Second, we show that wall-crossing functors commute with restriction
and induction functors, Section \ref{SS_WC_vs_Res}.
These are  crucial tools to compute these bijections
in the cases we need.

In Section \ref{S_cyclot} we recall a few additional facts about categories
$\mathcal{O}_c(W)$ for $W=G(\ell,1,n)$. In Section \ref{SS_cat_act}
we recall categorical Kac-Moody (\cite{Shan} and \cite{GM})
and Heisenberg, \cite{SV}, actions on the cyclotomic categories
$\mathcal{O}$, as well as the crystal for the Kac-Moody
action, \cite{cryst}.  Finally, Section \ref{SS_equi} recalls
a decomposition of  cyclotomic categories $\mathcal{O}$ that is used
to reduce the computation of supports to some special parameters $c$.

In Section \ref{S_main_proofs} we prove results explained in Section \ref{SS_main_res}.
In Section \ref{SS_Heis_cryst}, we introduce an $\slf_\infty$-crystal
on $\Part_\ell$ and establish some of its basic properties. We compute
the corresponding crystal operators in many chambers including asymptotic ones
in Section \ref{SS_comput_Heis}.
Then we show that the wall-crossing bijections commute with the Kac-Moody
and Heisenberg crystal operators, Section \ref{SS_wc_vs_cryst}. Next, we explain how to compute
the wall-crossing bijections through hyperplanes, Section \ref{SS_wc_bij_comput}.
 We summarize the computations of supports and of
Heisenberg crystals in Section \ref{SS_comput_summ} and give
an example of computation in Section \ref{SS_comput_example}.

Section \ref{S_App} is an appendix describing various developments related
to the main body of the paper.  In Section \ref{S_App_kappa0} we consider
the computation of supports in the case when $\kappa=0$. In Section
\ref{S_gen_groups} we explain how to reduce a computation of supports  for the complex reflection groups $G(\ell,r,n)$
to that for $G(\ell,1,n)$. Finally,
in Section \ref{S_commut_cryst} we explain a conjectural crystal version of the
level-rank duality for affine type A Kac-Moody algebras (involving three commuting
crystals)  and its categorical meaning\footnote{After the present paper appeared,
the problem of showing that the three crystals commute was solved by Gerber
in \cite{Gerber} using combinatorial methods.}. It is based on
a variant of techniques used in \cite{RSVV} and \cite{VV_proof}
combined with an approach of Bezrukavnikov and Yun, \cite{BY}
to Koszul duality for Kac-Moody groups.

{\bf Acknowledgements}. I would like to thank  Roman Bezrukavnikov for
many stimulating discussions. Also I am grateful to Nicolas Jacon, Emily Norton and
Seth Shelley-Abrahamson for numerous comments on the preliminary version of this text. 
I would also like to thank the referee for their careful reading of the paper and 
helpful comments. This work was
partially supported by the NSF under grants DMS-1161584, DMS-1501558.

\section{Cherednik algebras and their categories $\mathcal{O}$}\label{S_RCA}
\subsection{Rational Cherednik algebras}\label{SS_RCA}
Let $W$ be a complex reflection group and $\h$  its reflection representation. For a reflection hyperplane
$H$, the pointwise stabilizer $W_H$ is cyclic, let $\ell_H$ be the order of this group. The set of
the reflection hyperplanes will be denoted by $\mathfrak{H}$. Let $\alpha_H,\alpha^\vee_H$ denote
the eigenvectors for $W_H$ in $\h^*,\h$ with non-unit eigencharacters, partially normalized
by $\langle\alpha_H,\alpha_H^\vee\rangle=2$.
For a complex reflection $s$ we write $\alpha_s,\alpha_s^\vee$ for $\alpha_H,\alpha_H^\vee$
where $H=\h^s$.   Let $c:S\rightarrow \C$
be a function constant on the conjugacy classes. The space of such functions is denoted
by $\param$, it is a vector space of dimension $|S/W|$.

By definition, \cite[Section 1.4]{EG}, \cite[Section 3.1]{GGOR},
the rational Cherednik algebra  $H_c(=H_c(W)=H_c(W,\h))$ is the quotient of $T(\h\oplus \h^*)\# W$
by the following relations:
$$[x,x']=[y,y']=0, \,\, [y,x]=\langle y,x\rangle-\sum_{s\in S}c(s)\langle x,\alpha_s^\vee\rangle\langle y,\alpha_s\rangle s, \quad x,x'\in \h^*, y,y'\in \h.$$

\subsubsection{Deformation}
We would like to point out that $H_c$ is the specialization to
$c$  of a $\C[\param]$-algebra $H_{\param}$ defined as follows. The
space $\param^*$ has basis ${\bf c}_s$ naturally indexed
by the conjugacy classes of reflections. Then $H_{\param}$ is the quotient of $T(\h\oplus \h^*)\# W\otimes \C[\param]$
by the relations similar to the above but where  $c(s)\in \C$ is replaced with ${\bf c}_s\in \param^*$.
For a commutative algebra $R$ with a $W$-invariant map $c:S\rightarrow R$ we can consider the algebra $H_{R,c}=R\otimes_{\C[\param]}H_{\param}$. If $R=\C[\param^1]$ for an affine subspace $\param^1\subset \param$,
then we write $H_{\param^1}$ instead of $H_{R,c}$.

\subsubsection{PBW property and triangular decomposition}
Let us recall some structural results about $H_c$.
The algebra $H_c$ is filtered with $\deg \h^*=0$, $\deg W=0, \deg \h=1$. The associated graded algebra is $S(\h\oplus \h^*)\#W$, \cite[Section 1.2]{EG}. This yields
the triangular decomposition $H_c=S(\h^*)\otimes \C W\otimes S(\h)$, \cite[Section 3.2]{GGOR}.
The algebra $H_{\param}$ is also filtered  with $\deg\param^*=1$.
We get $H_{\param}=S(\h^*)\otimes \C[\param]W\otimes S(\h)$
as a $\C[\param]$-module.

\subsubsection{Euler element}
There is an {\it Euler element} $h\in H_c$ satisfying $[h,x]=x, [h,y]=-y, [h,w]=0$. It is constructed as follows.
Pick a basis $y_1,\ldots,y_n\in \h$ and let $x_1,\ldots,x_n\in \h^*$ be the dual basis. For $s\in S$, let
$\lambda_s$ denote the eigenvalue of $s$ in $\h^*$ different from $1$. Then
\begin{equation}\label{eq:Euler}
h:=\sum_{i=1}^n x_i y_i+\frac{n}{2}-\sum_{s\in S} \frac{2c(s)}{1-\lambda_s}s.\end{equation}

\subsubsection{Spherical subalgebras}\label{SSS_spherical}
Consider the averaging idempotent $e:=|W|^{-1}\sum_{w\in W}w\in \C W\subset H_c$. The {\it spherical subalgebra}
by definition is $eH_ce$.
More generally, let $\chi$ be a one-dimensional character of $W$. Let $e_\chi$ be the corresponding idempotent
in $\C W$. Form the algebra $e_\chi H_c e_\chi$.

We can also consider spherical subalgebras over $\param$, we get the algebras $e_\chi H_{\param} e_\chi$.
These algebras inherit the filtration from $H_{\param}$, the associated graded coincides
with $e_\chi \gr H_{\param} e_\chi$. Let $Z_{\param}$ denote the center of $\gr H_{\param}$.
It was shown by Etingof and Ginzburg in \cite[Theorem 3.1]{EG} that the map $z\mapsto ze_\chi$ defines
an isomorphism $Z_\param \rightarrow \gr e_\chi H_{\param} e_\chi$. In particular, the associated
graded algebras of $e_\chi H_{\param} e_\chi$ are all identified.

It turns out that $e_\chi H_{\param} e_\chi\cong e H_{\param} e$, where the isomorphism
induces a shift by an element  $\bar{\chi}\in \param$ on $\param$. The element $\bar{\chi}$
is constructed as follows. We can find elements $h_{H,j}\in \C$  with $j=0,\ldots,\ell_{H}-1$
and $h_{H,j}=h_{H',j}$ for $H'\in WH$  such that
\begin{equation}\label{eq:c_to_h}c(s)=\sum_{j=1}^{\ell-1}\frac{1-\lambda_s^j}{2}(h_{\h^s,j}-h_{\h^s,j-1})\end{equation}
Clearly, for fixed $H$, the numbers $h_{H,0},\ldots, h_{H,\ell_H-1}$ are defined up to a common summand.
We can recover the elements $h_{H,i}$ by the formula
\begin{equation}\label{eq:h_to_c} h_{H,i}=\frac{1}{\ell_H}\sum_{s\in W_H\setminus \{1\}}\frac{2c(s)}{\lambda_s-1}\lambda_s^{-i}
\end{equation}
Note that $\sum_{i=0}^{\ell_H-1}h_{H,i}=0$. We will view $h_{H,i}$ as an element of $\param^*$
whose value on $c:S\rightarrow \C$ is given by (\ref{eq:h_to_c}).

There is a homomorphism  $\operatorname{Hom}(W,\C^\times)
\rightarrow \prod_{H\in \mathfrak{H}/W} \operatorname{Irr}(W_H)$  given by
the restriction. It turns out that this map is an isomorphism, see \cite[3.3.1]{rouqqsch}. So to an arbitrary
$W$-invariant collection of elements $(a_H)$ with $0\leqslant a_H\leqslant \ell_H-1$ we can assign the character of $W$
that sends $s$ to $\lambda_s^{-a_H}$. To a character $\chi$ given in this form we assign the element
$\bar{\chi}\in \param$  by $h_{H,i}(\bar{\chi})=1-\frac{a_H}{\ell_H}$ if $i\geqslant \ell-a_H$
and $-\frac{a_H}{\ell_H}$ if $i<\ell-a_H$.

\begin{Lem}\label{Lem:spher_iso}
There is an isomorphism
$\iota: e H_{\param}e\xrightarrow{\sim} e_\chi H_{\param} e_{\chi}$ of filtered $\C$-algebras
that is the identity on the associated graded algebras and  maps $p\in \param^*$ to
$p+\langle\bar{\chi},p\rangle$.
\end{Lem}
\begin{proof}
The isomorphism is constructed in \cite[Proposition 5.6]{BC} (for a specialized parameter,
but our case is similar).
%
\end{proof}

\subsection{Highest weight categories}\label{SS_hw_cat}
In this section we recall highest weight categories.

Let $\Cat$ be a $\C$-linear abelian category equivalent to the category of modules
over some finite dimensional associative $\C$-algebra. Let $\Lambda$ be an indexing set for the
simples in $\Cat$, for $\tau\in \Lambda$, we write $L(\tau)$ for the simple object indexed by $\tau$ and $P(\tau)$
for its projective cover. By a highest weight category we mean a triple $(\Cat,\leqslant, \{\Delta(\tau)\}_{\tau\in \Lambda})$, where $\leqslant$ is a partial order on $\Lambda$ and $\Delta(\tau), \tau\in \Lambda,$
is a collection of {\it standard} objects in $\Cat$ satisfying the following conditions:
\begin{itemize}
\item[(i)] $\Hom_{\Cat}(\Delta(\tau),\Delta(\tau'))\neq 0$ implies $\tau\leqslant \tau'$.
\item[(ii)] $\End_{\Cat}(\Delta(\tau))=\C$.
\item[(iii)] There is an epimorphism $P(\tau)\twoheadrightarrow \Delta(\tau)$ whose kernel admits
a filtration with successive quotients of the form $\Delta(\tau')$ with $\tau'>\tau$.
\end{itemize}

\subsubsection{Costandard and tilting objects}
Recall that in any highest weight category
$\Cat$ one has costandard objects $\nabla(\tau), \tau \in \Lambda,$
with $\dim \Ext^i(\Delta(\tau),\nabla(\xi))=\delta_{i,0}\delta_{\tau,\xi}$.

By a tilting object in $\Cat$ we mean an object that is both standardly
filtered (=admits a filtration with standard quotients) and costandardly
filtered. The indecomposable tiltings are in bijection with $\Lambda$.
By a tilting generator we mean a tilting that contains every indecomposable
tilting as a summand.

\subsubsection{Highest weight subcategories}
Let $\Lambda_0$ be a poset ideal in $\Lambda$ (i.e., a subset such that, for each
$\lambda\in \Lambda_0, \lambda'\leqslant \lambda$, we have $\lambda'\in \Lambda_0$).
Consider the Serre subcategory
$\Cat(\Lambda_0)\subset \Cat$ spanned by the simples $L(\tau), \tau\in \Lambda_0$.
This is a highest weight category with respect to the order restricted from $\Lambda$
and with standard objects $\Delta(\tau), \tau\in \Lambda_0$ (and costandard
objects $\nabla(\tau),\tau\in \Lambda_0$). Moreover, the natural functor
$D^b(\Cat(\Lambda_0))\rightarrow D^b(\Cat)$ is a full embedding so that
$D^b(\Cat(\Lambda_0))$ gets identified with the full subcategory
$D^b_{\Cat(\Lambda_0)}(\Cat)$ of all objects with homology in $\Cat(\Lambda_0)$.
As usual, $D^b(\Cat/\Cat(\Lambda_0))$ gets identified with
$D^b(\Cat)/D^b(\Cat(\Lambda_0))$.

\subsubsection{Ringel duality}
Now recall the Ringel duality. Let $\Cat$ be a highest weight category
and let $T$ be a tilting generator. Set $\,^\vee\Cat:=\End_{\Cat}(T)^{opp}\operatorname{-mod}$.
Then $\,^\vee\Cat$ is a highest weight category with standard objects
$\,^\vee\Delta(\tau):=\Hom(T,\nabla(\tau))$. The sets $\Irr(\,^\vee\Cat)$ and
$\Irr(\Cat)$ are identified and the orders on them are opposite.
The functor $\mathcal{R}:=R\Hom_{\Cat}(T,\bullet)$ is a derived
equivalence $D^b(\Cat)\xrightarrow{\sim}D^b(\,^\vee\Cat)$ called
the Ringel duality functor. Note that $\,^\vee(\,^\vee\Cat)$ is naturally
identified with $\Cat^{opp}$. We write $\Cat^\vee$ for $(\,^\vee\Cat)^{opp}$
so that $\,^\vee (\Cat^\vee)=\Cat$. So we get a derived equivalence
$\mathcal{R}^{-1}:D^b(\Cat)\xrightarrow{\sim} D^b(\Cat^\vee)$
that maps $\Delta(\tau)$ to $\nabla^\vee(\tau)$.

\subsection{Categories $\mathcal{O}$}\label{SS_RCA_O}
Following \cite[Section 3.2]{GGOR}, we consider the full subcategory $\OCat_c(W)$ of $H_c\operatorname{-mod}$
consisting of all modules $M$ that are finitely generated over $S(\h^*)$ and such that
$\h$ acts on $M$ locally nilpotently.
For example, pick an irreducible representation $\tau$ of $W$. Then the {\it Verma
module} $\Delta_c(\tau):=H_c\otimes_{S(\h)\#W}\tau$ (here $\h$ acts by $0$ on $\tau$)
is in $\mathcal{O}_c(W)$.

\subsubsection{Supports}
To a module  $M\in\OCat_c(W)$ we can assign its support $\Supp(M)$ that, by definition, is the support of $M$
(as a coherent sheaf)  in $\h$. Clearly, $\Supp(M)$ is a closed $W$-stable subvariety.
For a parabolic subgroup $\underline{W}\subset W$, set $X(\underline{W})=W \h^{\underline{W}}$.
The support of $M$ is the union of some subvarieties $X(\underline{W})$. Moreover,
if $M$ is simple, then $\Supp(M)=X(\underline{W})$ for some $\underline{W}$.
See \cite[Section 3.8]{BE} for the proofs.

\subsubsection{Highest weight structure}
Let us describe a highest weight structure on $\OCat_c(W)$, \cite[Theorem 2.19]{GGOR}.
For the standard objects we take the Verma modules. A partial order on $\Lambda=\Irr(W)$
is introduced as follows. The element $\displaystyle \sum_{s\in S} \frac{2c(s)}{\lambda_s-1}s\in \C W$
is central so acts by a scalar, denoted by $c_\tau$ (and called the $c$-function), on $\tau$. 
We set $\tau<\xi$ if $c_\tau-c_\xi\in \Q_{>0}$ (we could take the coarser order by requiring the difference to lie in $\Z_{>0}$ but we do not need this). We write $<_c$ if we want to indicate the dependence on the parameter $c$.

The following is established in \cite[Section 4.3]{GGOR}.

\begin{Lem}\label{Lem:titl_support}
Let $L\in \Irr(\mathcal{O}_c(W))$ and let $T$ be a tilting generator in $\mathcal{O}_c(W)$.
Then $\dim\h-\dim \Supp(L)$ coincides with the minimal number $i$
such that $\Ext^i_{\mathcal{O}_c(W)}(T,L)\neq 0$.
\end{Lem}

\subsubsection{$K_0$ and characters}
We identify $K_0(\OCat_c(W))$ with $K_0(W\operatorname{-mod})$ by sending the
class $[\Delta_c(\tau)]$ to $[\tau]$.

Further, to a module $M\in\OCat_c(W)$ we can assign its {\it character}
$\mathsf{ch}(M):=\sum_{a\in \C}[M_a]_Wq^a$, where $M_a$ is the generalized
eigenspace for the action of $h$ with eigenvalue $a$, and $[M_a]_W$ is the class of $M_a$
in $K_0(W\operatorname{-mod})$.

We have the following easy lemma.

\begin{Lem}\label{Lem:K0_class_recover}
Let $M,M'$ be two objects in $\OCat_c(W)$ such that the coefficients
of $q^a$ in $\mathsf{ch}(M)$ and $\mathsf{ch}(M')$ coincide for any
$a$ of the form $c_\tau, \tau\in \Irr(W)$. Then $[M]=[M']$.
\end{Lem}

Here and below we write $[M]$ for the class of $M$ in $K_0(\OCat_c(W))$.

\subsubsection{Twist by a character of $W$}
Now let $\chi$ be a one-dimensional character of $W$.
Given $c\in \param$, define $c^\chi\in \param$ by $c^\chi(s)=\chi(s)^{-1}c(s)$.
We have an isomorphism $\psi_\chi: H_c\xrightarrow{\sim}H_{c^\chi}$ given on
the generators by $x\mapsto x, y\mapsto y, w\mapsto \chi(w)w$. This gives rise
to an equivalence $\psi_{\chi*}:\OCat_c(W)\xrightarrow{\sim}\OCat_{c^\chi}(W)$
that maps $\Delta_c(\tau)$ to $\Delta_{c^\chi}(\chi\otimes \tau)$.

\subsubsection{Deformation}
To finish this section, let us note that one can also define the category $\OCat_{R,c}(W)$ for a commutative algebra $R$: it consists of all $H_{R,c}$-modules that are finitely generated over $R\otimes_{\C}S(\h^*)$ and
have a locally nilpotent action of $\h$. If $R=\C[\param']$ for an affine subspace $\param'\subset \param$,
then we write $\OCat_{\param'}(W)$ instead of $\OCat_{R,c}(W)$.

\subsection{Harish-Chandra bimodules}\label{SS_HC_bimod}
Let us introduce Harish-Chandra (shortly, HC) bimodules following
\cite{BEG}. We say that an $H_{c'}$-$H_c$-bimodule $M$ is HC
if it is finitely generated and the adjoint actions
of $S(\h)^W$ and $S(\h^*)^W$ (that are subalgebras in both
$H_c,H_{c'}$) are locally nilpotent. For example, the algebra
$H_c$ is a HC $H_c$-bimodule.

An equivalent definition is as follows: a bimodule $\B$
is HC if and only if it has a bimodule filtration
such that $\gr\B$ is a finitely generated bimodule
over $S(\h\oplus\h^*)^W$ and the left and the right actions
of $S(\h\oplus \h^*)^W$ on $\gr\B$ coincide
(such a filtration is called good). In other words,
$\gr\B$ is a finitely generated $S(\h\oplus \h^*)^W$-module.
The equivalence of these two definitions was established
in \cite[Section 5.4]{sraco}.

\subsubsection{Shift bimodules}\label{SSS_HC_shift}
A further example is provided by shift (a.k.a. translation) bimodules introduced in
complete generality in \cite[Section 5]{BC}. Let $\chi$ be a character of $W$.
Recall  $e_\chi\in \C W$ and $\bar{\chi}\in \param$ that have appeared
in \ref{SSS_spherical}.
We get an $H_{c+\bar{\chi}}$-$H_c$ bimodule $$\mathcal{B}_{c,\bar{\chi}}:=
H_{c+\bar{\chi}}e\otimes_{eH_{c+\bar{\chi}}e}e H_{c+\bar{\chi}} e_{\chi}\otimes_{eH_ce}eH_c.$$
Similarly, we get the $H_{c}$-$H_{c+\bar{\chi}}$ bimodule $\mathcal{B}_{c+\bar{\chi},-\bar{\chi}}$.
These bimodules are HC by \cite[Section 3.1]{rouq_der}.

\subsubsection{Tensor products}
The following result is obtained in
\cite[Section 3.4]{rouq_der}.

\begin{Prop}\label{Prop:tensor}
Let $\B_1$ be a HC $H_{c''}$-$H_{c'}$-bimodule, $\B_2$ be a HC $H_{c'}$-$H_c$-module,
and $M\in \OCat_c(W)$. Then the following is true.
\begin{enumerate}
\item $\Tor_i^{H_{c'}}(\B_1,\B_2)$ is a HC $H_{c''}$-$H_c$-bimodule for any $i$.
\item $\Tor_i^{H_{c}}(\B_2,M)\in \OCat_{c'}$ for any $i$.
\item If $M$ is projective, then $\Tor^{H_{c}}_i(\B_2,M)=0$ for any $i>0$.
\end{enumerate}
\end{Prop}

So, for a HC $H_{c'}$-$H_c$-bimodule $\B$, we get a functor $\B\otimes^L_{H_c}\bullet:
D^b(\OCat_c(W))\rightarrow D^b(\OCat_{c'}(W))$.

\subsubsection{Deformations and supports in $\param$}\label{SSS_param_supp}
Let $\psi\in \param$. By a HC $(H_{\param},\psi)$-bimodule we mean a finitely
generated $H_{\param}$-bimodule $\B$ with locally nilpotent adjoint actions of
$S(\h)^W, S(\h^*)^W$ and such that $[z,b]=\langle z,\psi\rangle b$
for any $z\in \param^*, b\in \B$. Let $\HC(H_\param,\psi)$
denote the category of HC $(H_{\param},\psi)$-bimodules.

Note that $\B\otimes_{\C[\param]}\C_c$
is a HC $H_{c+\psi}$-$H_c$-bimodule. Conversely, any HC $H_{c'}$-$H_c$-bimodule
belongs to $\HC(H_{\param},c'-c)$. Similarly to \ref{SSS_HC_shift}, we have
bimodules   $\mathcal{B}_{\param,\bar{\chi}}\in \HC(H_{\param},\bar{\chi}),
\mathcal{B}_{\param+\bar{\chi},-\bar{\chi}}\in \HC(H_{\param},-\bar{\chi})$.

For $\B\in \HC(H_{\param},\psi)$ we define its {\it right $\param$-support}
$\Supp^r_{\param}(\B)$ as the set of all $c\in \param$ such that
$\B\otimes_{\C[\param]}\C_c\neq \{0\}$. Completely analogously to
\cite[Proposition 2.6]{CWR}, we see that $\Supp^r_{\param}(\B)$
is closed.

\subsubsection{HC bimodules over spherical subalgebras}
We can define the notion of a HC $A$-$A'$-bimodule,
where $A,A'$ are one of the algebras $H_{\param}, e_\chi H_{\param}e_\chi$
similarly to the above. Note that if $\B$ is a HC $(H_{\param},\psi)$-bimodule,
then $\B e_\chi$ is a HC $H_{\param-\psi}$-$e_\chi H_{\param} e_\chi$-bimodule.
As before, the tensor product of two HC bimodules is again a HC bimodule.

\subsection{Isomorphisms of completions}\label{SS_compl}
Here we are going to study completions of $H_\param$ and their connections
to the algebras $H_\param(\underline{W},\h)$, where $\underline{W}\subset W$
is a parabolic subgroup (the stabilizer of a point in $\h$).

We pick $b\in \h$. 
Set $H_\param^{\wedge_b}:=\C[\h/W]^{\wedge_b}\otimes_{\C[\h/W]}H_\param$.
This is a filtered algebra (with filtration inherited from $H_\param$)
containing $\C[\h/W]^{\wedge_b},H_\param$
as subalgebras. We can form the RCA $H_\param(\underline{W},\h):=
\C[\param]\otimes_{\C[\underline{\param}]}H_{\underline{\param}}(\underline{W},\h)$
for $\underline{W}$ acting on $\h$ (here $\underline{\param}$ is the parameter
space for $\underline{W}$). Form the completion $H_\param^{\wedge_b}(\underline{W},\h):=
\C[\h/\underline{W}]^{\wedge_b}\otimes_{\C[\h/\underline{W}]}H_\param(\underline{W},\h)$.

It turns out that there is an isomorphism of $H_\param^{\wedge_b}$
and the matrix algebra of size $|W/\underline{W}|$ over $H_\param^{\wedge_b}(\underline{W},\h)$,
\cite{BE}. We will need an invariant realization of the matrix algebra in terms of
centralizer algebras from \cite[Section 3.2]{BE}.

Let $A$ be an algebra equipped with a homomorphism from $\C\underline{W}$. Consider the space
$\operatorname{Fun}_{\underline{W}}(W,A)$ of all functions satisfying $f(uw)=uf(w)$ for all
$u\in \underline{W}, w\in W$. This is a free right $A$-module of rank $|W/\underline{W}|$.
Its endomorphism algebra, the centralizer algebra from \cite{BE}, will be denoted by $Z(W,\underline{W},A)$.  Note that $e Z(W,\underline{W},A)e=
\underline{e} A\underline{e}$, where we write $\underline{e}$ for the trivial idempotent in
$\C\underline{W}$. More generally, $e_\chi Z(W,\underline{W},A) e_\chi= \underline{e}_\chi A
\underline{e}_\chi$ for any one-dimensional character $\chi$ of $W$.

The algebras $\gr H_{c}^{\wedge_b}=\C[\h/W]^{\wedge_b}\otimes_{\C[\h/W]}\C[\h\oplus\h^*]\#W$
and $Z(W,\underline{W}, \gr H_{c}^{\wedge_b})$ are naturally isomorphic via
$\theta^0: \gr H_c^{\wedge_b}\xrightarrow{\sim} Z(W,\underline{W}, \gr H_{c}^{\wedge_b})$  given by the following formulas:
\begin{equation}
\begin{split}
&[\theta^0(F)f](w)=F f(w),\\
&[\theta^0(u)f](w)=f(wu),\\
&[\theta^0(v)f](w)=(wv)f(w),\\
&F\in \C[\h/W]^{\wedge_b}, u\in W, v\in \h\oplus \h^*.
\end{split}
\end{equation}
In particular, $\theta^0$ restricts to an isomorphism $\C[\h/W]^{\wedge_b}\xrightarrow{\sim}
\C[\h/\underline{W}]^{\wedge_b}$.

\begin{Prop}\label{Prop:compl_iso1}
The following is true.
\begin{enumerate}
\item
There is a $\C[\param]$-linear isomorphism $\theta_b: H_{\param}^{\wedge_b}\xrightarrow{\sim} Z(W,\underline{W}, H_{\param}^{\wedge_b}(\underline{W},\h))$
of filtered algebras that gives $\theta^0$ after passing to the associated graded and
specializing to $0\in \param$.
\item If $\theta'_b$ is another
isomorphism with these properties, then there is an invertible element
$F\in \C[\h/W]^{\wedge_b}\otimes (\C\oplus \param^*)$ such that $\theta_b':=\theta_b\circ
\exp(\operatorname{ad}F)$.
\end{enumerate}
\end{Prop}
\begin{proof}
The isomorphism $\theta_b$ is constructed explicitly in \cite[Section 3.3]{BE}. The existence of $F$ is
proved in the same way as in  \cite[Lemma 5.2.1]{sraco}.
\end{proof}

The isomorphism $\theta_b$ restricts to $e_\chi H_{\param}^{\wedge_b} e_\chi
\xrightarrow{\sim} \underline{e}_\chi H_{\param}^{\wedge_b}(\underline{W},\h)\underline{e}_\chi$
and a straightforward analog of (2) holds. We note that
$e_\chi H_{\param}^{\wedge_b} e_\chi=\C[\h/W]^{\wedge_b}\otimes_{\C[\h/W]}e_\chi H_{\param}e_\chi$.



\subsection{Induction and restriction functors}\label{SS_IndRes}
An isomorphism of completions in Section \ref{SS_compl} allows one to define
restriction functors $\Res^W_{\underline{W}}: \OCat_c(W)\rightarrow \OCat_c(\underline{W}),
\bullet_{\dagger}: \HC(H_{\param},\psi)\rightarrow \HC(H_{\param}(\underline{W}),\psi)$
and the induction functor $\Ind^W_{\underline{W}}:\OCat_c(\underline{W})\rightarrow
\OCat_c(W)$. The functors for the categories $\mathcal{O}$ were introduced
in \cite[Section 3.5]{BE}, the case of HC bimodules was treated in \cite[Section 3]{sraco}.
We recall the constructions below in this section.

\subsubsection{Functors for categories $\mathcal{O}$: construction}\label{SSS_ResO_constr}
We start by explaining how to construct the functors $\Res^W_{\underline{W}},\Ind^W_{\underline{W}}$.

Let $\OCat^{\wedge_b}_c$ denote the category of all $H_c^{\wedge_b}$-modules that are finitely
generated over $\C[\h/W]^{\wedge_b}$. This category is equivalent to $\OCat_c(\underline{W})$.
The equivalence is established as a composition of several intermediate equivalences.
First, note that $\OCat_c(\underline{W})$ is equivalent to $\OCat_c^{\wedge_b}(\underline{W},\h)$
via $N\mapsto \C[\h/\underline{W}]^{\wedge_b}\otimes_{\C[\underline{\h}/\underline{W}]} N, N\in \OCat_c(\underline{W})$.
Next, $\OCat^{\wedge_b}_c$ is equivalent to $\OCat_c^{\wedge_b}(\underline{W},\h)$
via $M'\mapsto e(\underline{W})\theta_{b*}(M')$, where $e(\underline{W})$ is the idempotent
in $Z(W,\underline{W}, H_c^{\wedge_b}(\underline{W},\h))$ given by $[e(\underline{W})f](w)=f(w)$
if $w\in \underline{W}$ and $[e(\underline{W})f](w)=0$, else. Let $\mathcal{F}$
denote the resulting equivalence $\OCat_c^{\wedge_b}\xrightarrow{\sim}\OCat_c(\underline{W})$.

We have an exact functor $\OCat_c\rightarrow \OCat_c^{\wedge_b}$ given by $M\mapsto M^{\wedge_b}:=
\C[\h/W]^{\wedge_b}\otimes_{\C[\h/W]}M$. We set $\Res_{\underline{W}}^W:=\mathcal{F}(\bullet^{\wedge_b})$,
this functor is independent of $b$ up to an isomorphism, see \cite[Section 3.7]{BE}.
It was shown in \cite[Section 3.5]{BE} that it admits
an exact right adjoint functor, the induction functor $\Ind_{\underline{W}}^W:\OCat_c(\underline{W})
\rightarrow \OCat_c(W)$.

Note that the functors $\Res^W_{\underline{W}}, \Ind_{\underline{W}}^W$ do not depend
(up to an isomorphism) on the choice of $\theta_b$, this follows from (2) of Proposition \ref{Prop:compl_iso1}.

\subsubsection{Functors for categories $\mathcal{O}$: properties}
The functor $\Ind_{\underline{W}}^W$ is also left adjoint to $\Res^W_{\underline{W}}$, see \cite[Section 2.4]{Shan}
for the proof of this under some restrictions on $W$ and  \cite{fun_iso} in general.
On the level of $K_0$'s, these functors behave like the restriction
and the induction for groups, \cite[Section 3.6]{BE}.

The functors $\Ind$ and $\Res$ are compatible with the supports as follows.
Let $L\in \OCat_c(W)$ be a simple module and $\Supp L=X(W_1)$.
Then $\Supp \Res^W_{\underline{W}} L=\bigcup_{W_1'}\underline{X}(W_1')$. Here
$W_1'$ runs over all parabolic subgroups of $\underline{W}$
that are conjugate to $W_1$ in $W$. We write $\underline{X}(W_1')$
for $\underline{W} \underline{\h}^{W_1'}$, where $\underline{\h}:=
\h/\h^{\underline{W}}$.

\begin{Rem}\label{Rem:characters}
If one knows the classes of $L_c(\tau)$ in $K_0(\OCat_c(W))=K_0(W\operatorname{-mod})$,
then, in principle,  one can use the properties above to compute the support of $L_c(\tau)$.
For example, if $W=G(\ell,1,n)$, then the classes were computed in
\cite{RSVV,VV_proof,Webster_new}. However, the formulas are very involved
and so computing the support in this way is  much  more
complicated than what is proposed in the present paper.
\end{Rem}

\begin{Lem}\label{Lem:Ind_support}
Let $\underline{L}$ be a simple module in the category $\OCat_c(\underline{W})$
with $\Supp \underline{L}=\underline{X}(\underline{W}_1)$. Then $\Ind_{\underline{W}}^W(\underline{L})\neq 0$ and
the supports of any  submodule and any quotient module of  $\Ind_{\underline{W}}^W(\underline{L})$  are equal to $X(\underline{W}_1)$.
\end{Lem}
\begin{proof}
By \cite[Proposition 2.7]{SV}, $\Ind_{\underline{W}}^W(\underline{L})\neq 0$
and the support of $\Ind_{\underline{W}}^W(\underline{L})$ equals $X(\underline{W}_1)$.
Now let $M$ be a quotient of $\Ind_{\underline{W}}^W(\underline{L})$ whose support
is a proper subvariety  of $X(\underline{W}_1)$. Since $\Res$ is right adjoint to $\Ind$,
we get a nonzero homomorphism $\underline{L}\rightarrow  \Res_{\underline{W}}^W(M)$
and so $\underline{X}(\underline{W}_1)\subsetneq \operatorname{Supp}(\Res_{\underline{W}}^W(M))$.
This contradicts the description of the support of restriction given above.
To establish the claim for subs, we argue similarly using the fact
that $\Res$ is left adjoint to $\Ind$.
\end{proof}

\subsubsection{Functors for HC bimodules: construction}
Let us recall restriction functors for HC bimodules. According to \cite[Section 3.6]{sraco}, there is an exact
$\C[\param]$-linear functor $\bullet_{\dagger,\underline{W}}:
\HC(H_{\param}(W),\psi)\rightarrow \HC(H_{\param}(\underline{W}),\psi)$.

The functor  $\bullet_{\dagger,\underline{W}}$ is constructed as follows.
By a HC $H^{\wedge_b}_{c'}$-$H^{\wedge_b}_{c}$-bimodule we mean
a bimodule $\B'$ that is equipped with a filtration such that
$\gr\B'$ is a finitely generated $\C[\h/W]^{\wedge_b}\otimes_{\C[\h/W]}S(\h\oplus \h^*)^W$-module.
More generally, a HC $(H^{\wedge_b}_{\param},\psi)$-bimodule is
a bimodule $\B$ (with the same compatibility of left and right $\param^*$-actions
as before) that can be equipped with a bimodule filtration such that
$\gr\B$ is a finitely generated  $\C[\h/W]^{\wedge_b}\otimes_{\C[\h/W]}Z_{\param}$-module,
where we write $Z_{\param}$ for the center of $\gr H_{\param}$. Recall that
it follows from \cite[Theorem 3.1]{EG}, that $Z_{\param}$ is a graded free deformation of $S(\h\oplus \h^*)^W$.
So any HC $H^{\wedge_b}_{c'}$-$H^{\wedge_b}_{c}$-bimodule is also an HC
$(H^{\wedge_b}_{\param},\psi)$-bimodule.

We have an exact completion functor $\B\mapsto \B^{\wedge_b}:=\C[\h/W]^{\wedge_b}\otimes_{\C[\h/W]}\B$
from $\HC(H_{\param},\psi)$ to $\HC(H^{\wedge_b}_{\param},\psi)$. The categories
$\HC(H^{\wedge_b}_{\param},\psi)$ and $\HC(H^{\wedge_b}_{\param}(\underline{W},\h),\psi)$
are equivalent, this is analogous to an equivalence between the categories $\OCat$
recalled in \ref{SSS_ResO_constr}.

The categories $\HC(H^{\wedge_b}_{\param}(\underline{W},\h),\psi)$ and $\HC(H_{\param}(\underline{W},\underline{\h}),\psi)$
are equivalent as well. An equivalence is constructed as follows:
to $\underline{\B}\in \HC(H_{\param}(\underline{W},\underline{\h}),\psi)$ we assign
$$H_{\param}^{\wedge_b}(\underline{W},\h)\otimes_{H_{\param}(\underline{W},\underline{\h})}\underline{\B}.$$
A quasi-inverse equivalence sends $\B'$ to the $\underline{h}$-finite part of
the centralizer of $D(\h^{\underline{W}})\subset H_{\param}(\underline{W},\h)$
in $\B'$. Here $\underline{h}$ is the Euler element in $H_c(\underline{W})$.
The claim that these two functors are quasi-inverse equivalences
can be easily deduced from \cite[Section 5.5]{sraco}.

Let $\mathcal{F}$ denote the resulting equivalence $\HC(H^{\wedge_b}_{\param},\psi)\xrightarrow{\sim}
\HC(H_\param(\underline{W},\underline{\h}),\psi)$. We set $\B_{\dagger,\underline{W}}:=\mathcal{F}(\B^{\wedge_b})$.

The similar construction works for HC $H_{\param}$-$e_\chi H_{\param} e_\chi$-bimodules,
etc. Note that $(\B e_\chi)_{\dagger,\underline{W}}$ is naturally isomorphic to
$(\B_{\dagger,\underline{W}})\underline{e}_\chi$ (and the same is true for the multiplications
by $e_\chi$ from the left).

\subsubsection{Functors for HC bimodules: properties}\label{SSS_res_HC_prop}
The functor $\bullet_{\dagger,\underline{W}}$ has the following properties established
in \cite{sraco,rouq_der}.

1) The functor $\bullet_{\dagger,\underline{W}}$ intertwines $\Tor_i$, see
\cite[Lemma 3.11]{rouq_der}.

2) We have a bifunctorial isomorphism $\Res^W_{\underline{W}}(\B\otimes^L_{H_c}M)
\cong \B_{\dagger,\underline{W}}\otimes^L_{H_c(\underline{W})}\Res^W_{\underline{W}}M$,
see \cite[Section 5.5]{sraco}.

Let us recall the notion of the associated variety of a HC bimodule.
Equip $\B\in \HC(H_{\param},\psi)$
with a  good filtration.  Then for the associated variety, $\VA(\B)$,
we take the support of the $S(\h\oplus \h^*)^W$-module $\gr\B/ \param^* \gr\B$.

Now let us explain how  the restriction functors behaves on the associated varieties.

3) The associated variety of $\B_{\dagger,\underline{W}}$ can be recovered from that of
$\B$ as follows. Pick a point $v\in \h\oplus \h^*$ with stabilizer $\underline{W}$.
The formal neighborhood of $Wv$ in $(\h\oplus\h^*)/W$ is naturally identified with
the formal neighborhood of $0$ in $(\h\oplus
\h^*)/\underline{W}$. Then $\VA(\B_{\dagger,\underline{W}})$
is uniquely recovered from $$\VA(\B)^{\wedge_{Wv}}=\left((\h\oplus \h^*)^{\underline{W}}\times
\VA(\B_{\dagger,\underline{W}})\right)^{\wedge_0}.$$
This is easily seen from the construction.

\subsection{Chambers and walls}\label{SS_chamber}
We consider the $\Z$-lattice and the $\Q$-lattice $\param^*_\Z\subset \param^*_{\Q}\subset \param^*$
spanned by the elements $h_{H,i}-h_{H,j}$ and the dual lattices $\param_{\Z}\subset \param_{\Q}\subset \param$.
The lattice $\param_{\Z}$ is spanned by the elements $\bar{\chi}$.
We will need a certain sublattice  in  $\param_{\Z}$.  In \cite[Section 7.2]{BC},
Berest and Chalykh established a group homomorphism $\mathsf{tw}:\param_\Z\rightarrow
\operatorname{Bij}(\Irr W)$ called the {\it KZ twist}. Set $\underline{\param}_{\Z}:=\ker \mathsf{tw}$.

As we have seen in \cite[Lemma 2.6]{rouq_der}, the function $c\mapsto c_\tau$ is in $\param^*_{\Q}$.
Define an equivalence relation $\sim$ on $\Irr(W)$ by setting
$\tau\sim \tau'$ if $c_{\tau}=c_{\tau'}$ for every parameter $c$.  Now if $\tau\not\sim \xi$,
then we have the hyperplane $\Pi_{\tau,\xi}$ in $\param$ given by $c_{\tau}=c_{\xi}$.
All the hyperplanes $\Pi_{\tau,\xi}$  are rational.

Fix a coset $c+\underline{\param}_{\Z}$ and consider $c'$ in this coset.
We write $c\prec c'$
if $\tau\leqslant_c \xi$ implies $\tau\leqslant_{c'}\xi$. We write $c\sim c'$ if $c\prec c'$
and $c'\prec c$. The equivalence classes are relative interiors in the cones defined by the hyperplane
arrangement $\{\Pi_{\tau,\xi}\mid \tau\not\sim \xi, c_\tau-c_\xi\in \Q\}$
on $c+\underline{\param}_{\Z}$.


We are mostly interested in the
open cones. In what follows,  the open cones in this stratification will be called {\it open chambers}.
We often view them as subcones in $c+\param_{\Q}$.
For each open chamber we have its opposite chamber, where the order is opposite. Then we can talk about {\it faces} of the closures of the chambers. These are open cones in affine subspaces of $c+\param_{\Q}$.
We note that faces do not need to intersect $c+\underline{\param}_{\Z}$.

In what follows we will need the following definition.

\begin{defi}\label{defi:Weil_generic}
Let $X$ be an irreducible algebraic variety over $\C$. By a {\it Weil generic} point of
$X$ we mean a point lying outside a countable union of proper closed subvarieties.
\end{defi}

Note that if $c$ is Weil generic in $\param$,
we have just one open chamber, while for a Weil generic $c$ on a rational hyperplane parallel to $\Pi_{\tau,\xi}$
we have exactly two  open cones that are opposite to each other.

Here is an important result, \cite[Proposition 4.2]{rouq_der}, on a category equivalence.

\begin{Prop}\label{Prop:cat_equi}
Let $c,c'$ be such that $c'-c\in \underline{\param}_{\Z}$. Suppose $c\prec c'$.
Then there is a category equivalence $\OCat_c\xrightarrow{\sim}\OCat_{c'}$
mapping $\Delta_c(\tau)$ to $\Delta_{c'}(\tau)$ and preserving the supports
of the simple modules.
\end{Prop}

So, basically, we do not need to distinguish between parameters lying in the
same open chamber of $c+\underline{\param}_{\Z}$.

\subsection{Wall-crossing functors}\label{SS_WC}
In this section we recall wall-crossing functors introduced in \cite{rouq_der}.

\subsubsection{Construction}\label{SSS_WC_constr}
Let us construct a wall-crossing bimodule $\B_c(\psi)$ that is a
HC $H_{c-\psi}$-$H_c$-bimodule. We assume that $c$
lies in the interior of an open chamber $C$ of $c+\underline{\param}_{\Z}$.
Pick a face $F$ of (the closure of) $C$ and
choose $\psi\in \underline{\param}_{\Z}$ such that $c-\psi$ lies  in the
chamber $C^-$ that is opposite to $C$ with respect to $F$. This means that
$F$ is a face for both $C,C^-$ and there is an interval with end points
in $C,C^-$ and a mid-point in $F$. Let $F^0$ denote the associated
cone of $F$ in $\param_{\Q}$, it is defined by the linear parts of
the inequalities defining $F$.

Thanks to Proposition
\ref{Prop:cat_equi}, we may replace $c$ with a Zariski generic $c'\in \param$
with $c'-c\in F^0\cap \underline{\param}_{\Z}$ without changing the category $\mathcal{O}$.
Note that, by the construction, $F^0$ is rational, so
$F^0\cap \underline{\param}_{\Z}$ is non-empty.

Let $\param'$ denote
the affine subspace of $\param$ containing $c$, whose associated vector
space is spanned by $F^0$.

Consider a sequence of characters $\chi_1,\ldots,\chi_k$ with
$-\psi=\sum_{i=1}^k \epsilon_i \bar{\chi}_i$, where $\epsilon_i\in \{\pm 1\}$.
Then set $\param'_j:=\param'+\epsilon_1\bar{\chi}_1+\ldots+\epsilon_j \bar{\chi}_j$ and
$$\B_{\param',\psi}:=\B_{\param'_{k-1},\epsilon_k \bar{\chi}_k}\otimes_{H_{\param'_{k-1}}} \ldots\otimes_{H_{\param'_2}}\B_{\param'_1,\epsilon_2 \bar{\chi}_2}\otimes_{H_{\param'_1}}\B_{\param'_0,\epsilon_1 \bar{\chi}_1}.$$
This is an object in $\HC(H_{\param'},-\psi)$.

The construction in \cite[Section 5.2]{rouq_der} yields two sub-bimodules $\{0\}\subset
\B^1\subset \B^2\subset \B_{\param',\psi}$
with the property that, for a Weil generic $\hat{c}\in \param'$, the bimodule
$\B^2_{\hat{c}}/\B^1_{\hat{c}}$ is simple with full associated variety,
while $\B^1_{\hat{c}},\B_{\hat{c},\psi}/\B^2_{\hat{c}}$ have proper associated varieties.
We set $\B_{\param'}(\psi):=\B^2/\B^1$.

\begin{Lem}\label{Lem:wc_well_def}
The specialization $\B_c(\psi)$ is well-defined for a Zariski generic $c\in \param'$
meaning that for any other choice of sub-bimodules $\B'^1\subset \B'^2$ (satisfying
the assumptions before the lemma) we have
$\B^2_c/\B^1_c\cong \B'^2_c/\B'^1_c$ for a Zariski generic $c\in \param'$.
\end{Lem}
\begin{proof}
Set $\tilde{\B}^1=\tilde{\B}'^1:=\B^1+\B'^1, \tilde{\B}^2:=\B^2+\B'^1,
\tilde{\B}'^2:=\B'^2+\B^1$. The bimodules $\tilde{\B}_{\hat{c}}^1,\tilde{\B}_{\hat{c}}'^1,
\B_{\hat{c},\psi}/\tilde{\B}_{\hat{c}}^2,$ $ \B_{\hat{c},\psi}/\tilde{\B}_{\hat{c}}'^2$ still have
proper associated varieties. Moreover, $\B^2/\B^1\twoheadrightarrow \tilde{\B}^2/\tilde{\B}^1$
and $\B'^2/\B'^1\twoheadrightarrow \tilde{\B}'^2/\tilde{\B}'^1$. Since the Weil
generic fibers of $\B^2/\B^1, \B'^2/\B'^1$ are simple, it follows that these epimorphisms
are iso when specialized to $\hat{c}$. Therefore they are iso when specialized
to Zariski generic parameters (this is a consequence of the claim that the $\param$-supports
of HC bimodules are closed, see \ref{SSS_param_supp}). So we can replace $\B^1,\B'^1,
\B^2,\B'^2$ with $\tilde{\B}^1,\tilde{\B}'^1,\tilde{\B}^2,\tilde{\B}'^2$.

Thus we can assume that
$\B^1=\B'^1$. Now we can replace both $\B^2,\B'^2$ with $\B^2\cap \B'^2$
without changing $\B^2_c/\B^1_c$ and $\B'^2_c/\B'^1_c$ (for a Zariski generic $c$).
\end{proof}

For a Zariski generic $c\in \param'$, define the wall-crossing functor
$\WC_{c-\psi\leftarrow c}:=\B_c(\psi)\otimes^L_{H_c}\bullet:
D^b(\OCat_c) \rightarrow D^b(\OCat_{c-\psi})$.

\subsubsection{Properties}\label{SSS_WC_properties}
Now let us explain some important properties of wall-crossing functors.

The following is \cite[Proposition 5.3]{rouq_der}.

\begin{Prop}\label{Prop:der_equi}
For a Zariski generic $c\in \param'$, the functor $\WC_{c-\psi\leftarrow c}$
is a derived equivalence.
\end{Prop}

Further, it turns out that, for a Weil generic $\hat{c}\in \param'$, the
functor $\WC_{\hat{c}-\psi\leftarrow \hat{c}}$ realizes the inverse Ringel duality.
More precisely, the following holds, see \cite[Theorem 4.1]{rouq_der}.

\begin{Lem}\label{Lem:wc_Ringel}
There is a labelling preserving highest weight equivalence $\OCat_{\hat{c}-\psi}\xrightarrow{\sim}
\OCat_{\hat{c}}^\vee$ that intertwines $\WC_{\hat{c}-\psi\leftarrow \hat{c}}$ with the inverse Ringel
duality functor $D^b(\OCat_{\hat{c}})\xrightarrow{\sim} D^b(\OCat_{\hat{c}}^\vee)$. In particular,
$\WC_{\hat{c}-\psi\leftarrow \hat{c}}(\Delta_{\hat{c}}(\tau))=\nabla_{\hat{c}-\psi}(\tau)$.
\end{Lem}

The most important property of $\WC_{c-\psi\leftarrow c}$ is that it is a {\it perverse
equivalence}. To define a perverse equivalence, one needs filtrations by Serre
subcategories, in our case those will come from certain chains of two-sided ideals.
We have sequences of two-sided ideals $\{0\}
= I^{\param'}_n\subset I^{\param'}_{n-1}\subset \ldots
\subset I^{\param'}_0\subset H_{\param'}$ and
$\{0\}= I^{\param'-\psi}_n\subset I^{\param'-\psi}_{n-1}\subset \ldots
\subset I^{\param'-\psi}_0\subset H_{\param'-\psi}$ (here $n=\dim\h$) that have the following
property: for a Weil generic $\hat{c}\in \param'$ the specialization
$I^{\hat{c}}_i$ is the minimal ideal in $H_{\hat{c}}$ with the property
$\dim \VA(H_{\hat{c}}/I^{\hat{c}}_i)\leqslant 2i$ and the similar property
holds for $I^{\hat{c}-\psi}_i$. Similarly to Lemma \ref{Lem:wc_well_def},
this shows that the specializations
$I^c_i,I^{c-\psi}_i$ are well-defined for a Zariski generic parameter
$c\in \param'$. Moreover, for such a $c$, we have $(I^c_i)^2=I^c_i$
and $(I^{c-\psi}_i)^2=I^{c-\psi}_i$. So we can consider the Serre
subcategories $\Cat^c_i\subset \OCat_c(W), \Cat^{c-\psi}_i\subset \OCat_{c-\psi}(W)$
consisting of all modules annihilated by $I^c_{n-i}, I^{c-\psi}_{n-i}$.

The following claim is \cite[Theorem 6.1]{rouq_der}.

\begin{Prop}\label{Prop:perverse}
The equivalence $\WC_{c-\psi\leftarrow c}$ is perverse with respect to the
filtrations $\Cat^c_i\subset \OCat_c(W), \Cat^{c-\psi}_i\subset \OCat_{c-\psi}(W)$
meaning that the following holds.
\begin{itemize}
\item[(P1)] $\WC_{c-\psi\leftarrow c}$ restricts to an equivalence
between $D^b_{\Cat^c_i}(\OCat_c(W))$ and $D^b_{\Cat^{c-\psi}_i}(\OCat_{c-\psi}(W))$.
Here we write $D^b_{\Cat^c_i}(\OCat_c(W))$ for the full subcategory of
$D^b(\OCat_c(W))$ of all objects with homology in $\Cat^c_i$.
\item[(P2)] For $M\in \Cat^c_i$, we have $H_j(\WC_{c-\psi\leftarrow c}M)=0$
for $j<i$ and $H_{j}(\WC_{c-\psi\leftarrow c}M)\in \Cat^{c-\psi}_{i+1}$ for $j>i$.
\item[(P3)] The functor $M\mapsto H_i(\WC_{c-\psi\leftarrow c}M)$
defines an equivalence $\Cat^c_i/\Cat^c_{i+1}\xrightarrow{\sim}\Cat^{c-\psi}_i/\Cat^{c-\psi}_{i+1}$.
\end{itemize}
\end{Prop}

(P3) allows  us to define a natural bijection $\Irr(\Cat^c_i/\Cat^c_{i+1})\xrightarrow{\sim}
\Irr(\Cat^{c-\psi}_i/\Cat^{c-\psi}_{i+1})$. Since $$\Irr(W)=\bigsqcup_{i}\Irr(\Cat^c_i/\Cat^c_{i+1})=
\bigsqcup_i \Irr(\Cat^{c-\psi}_i/\Cat^{c-\psi}_{i+1}),$$ we get a self-bijection
$\wc_{c-\psi\leftarrow c}:\Irr(W)\xrightarrow{\sim}\Irr(W)$ to be called the
{\it wall-crossing bijection}. Below we will see that this bijection is actually independent
of $c$ as long as $c$ is Zariski generic.

\begin{Cor}\label{Cor:wc_bij_supp}
The bijection $\wc_{c-\psi\leftarrow c}$ preserves supports.
\end{Cor}
\begin{proof}
Let $\B$ be a HC $H_{c'}$-$H_c$ bimodule and $L\in \Irr(\OCat_c)$. We claim that
$\Supp(\operatorname{Tor}^{H_c}_i(\B,L))\subset \Supp(L)$. This follows from
the compatibility of Tor's with the restriction functors and from the compatibility
of the restriction functors with supports, see \ref{SSS_res_HC_prop}. It follows that
$\Supp(\wc_{c-\psi\leftarrow c}L)\subset \Supp L$. Recall that $\wc_{c-\psi\leftarrow c}$ is a bijection.
So, for each closed subvariety
$Y\subset \h$, the number of simples in
$\OCat_c$ whose support is contained in $Y$ does not exceed the similarly
defined number for $\OCat_{c-\psi}$. But we also have
$\Supp(\wc_{c\leftarrow c-\psi}L')\subset \Supp L'$ for every
simple $L'$ in $\OCat_{c-\psi}$. So the numbers of simples with given
support is the same for $\OCat_c, \OCat_{c-\psi}$. We see that
$\Supp(\wc_{c-\psi\leftarrow c}L)= \Supp L$.
\end{proof}

\section{Further properties of wall-crossing}\label{S_wc_further}
\subsection{Wall-crossing bijections are independent of the choice of $c$}\label{SS_wc_bij_gen}
The goal of this section is to prove the following claim.
\begin{Prop}\label{Prop:wc_bij_comp}
The bijection $\wc_{c-\psi\leftarrow c}:\Irr(W)\rightarrow \Irr(W)$
is independent of the choice of a Zariski generic $c$.
\end{Prop}
This will allow us to compute some of the wall-crossing bijections for the groups
$G(\ell,1,n)$ in Section \ref{SS_wc_bij_comput} below. We prove Proposition \ref{Prop:wc_bij_comp}
in \ref{SSS_wc_proof_compl} after some preliminary considerations.

\subsubsection{Wall-crossing on $K_0$}
We have the following.

\begin{Prop}\label{Prop:WC_K0}
Let  $c\in \param'$ be Zariski generic. Then the following is true.
\begin{enumerate}
\item The complex $\WC_{c-\psi\leftarrow c}\Delta_c(\tau)$
has no higher homology and its class in $K_0$ equals $[\nabla_{c-\psi}(\tau)]$.
\item In particular,
$\WC_{c-\psi\leftarrow c}$ gives the identity map $K_0(\OCat_c(W))\rightarrow K_0(\OCat_{c-\psi}(W))$
(recall that both are identified with $K_0(W\operatorname{-mod})$).
\end{enumerate}
\end{Prop}
\begin{proof}
Let us prove (1). By Lemma \ref{Lem:wc_Ringel},
when $\hat{c}$ is Weil generic, we have \begin{equation}\label{eq:WC_Delta_image}\WC_{\hat{c}-\psi\leftarrow \hat{c}}\Delta_{\hat{c}}(\tau)=\nabla_{\hat{c}-\psi}(\tau).\end{equation}
The objects $H_i(\B_{\param'}(
\psi)\otimes^L_{H_{\param'}}\Delta_{\param'}(\tau))$ are in $\OCat_{\param'}(W)$
for any $i$. In particular, they are finitely generated over  $\C[\h][\param']$
and hence are generically free over $\C[\param']$.
Together with (\ref{Lem:wc_Ringel}), this implies that $H_i(\WC_{c-\psi\leftarrow c}\Delta_c(\tau))=0$ for $i>0$
and Zariski generic $c$. It is easy to see that the modules $M=H_0(\WC_{c-\psi\leftarrow c}\Delta_c(\tau))$
and $M'=\nabla_{c-\psi}(\tau)$ satisfy the condition of Lemma \ref{Lem:K0_class_recover}
(this condition is preserved by degeneration from $\hat{c}$ to $c$). This implies the claim
about the classes in $K_0$.

(2) follows from (1) and the equality $[\nabla_{c-\psi}(\tau)]=[\Delta_{c-\psi}(\tau)]$
proved in \cite[Proposition 3.3]{GGOR}.
\end{proof}

\subsubsection{Pre-orders and classes of degenerations}
We will define a pre-order on $\Irr(W)$ refined by $\leqslant_c$. Namely, recall
that we have fixed a face $F$ of the chamber of $c$ and the associated
linear cone $F^0\subset \param_{\Q}$. Pick $c^0\in F^0\cap \underline{\param}_{\Z}$
to be Zariski generic (and in particular, sufficiently deep inside)
and replace $c$ with $c+c^0$. This leads
to the same pre-order $\leqslant_c$ and hence to the equivalent category
$\OCat_c(W)$, see Proposition \ref{Prop:cat_equi}. Let $\hat{c}$ denote a Weil generic
element in $\param'$.

We define a pre-order $\leqslant_F$ on $\Irr(W)$ by setting $\tau\leqslant_F \tau'$ if
$c^0_{\tau}-c^0_{\tau'}\in  \mathbb{Q}_{\geqslant 0}$. The choice of $c$ shows that
$\tau\leqslant_c \tau'$ implies $\tau\leqslant_F\tau'$. Moreover, if $\tau\sim_F \tau'$,
then $\tau\leqslant_c \tau'$ is equivalent to $\tau\leqslant_{\hat{c}}\tau'$.

We will need the following technical lemma.

\begin{Lem}\label{Lem:class_simple_degen}
Let $\tau\in \operatorname{Irr}(W)$. Then we have the following equality in $K_0(\mathcal{O}_c(W))=
K_0(\mathcal{O}_{\hat{c}}(W))$:
$$[L_{\hat{c}}(\tau)]=[L_c(\tau)]+\sum_{\xi<_F \tau}n_\xi [L_c(\xi)] $$
for some nonnegative integers $n_\xi$.
\end{Lem}
\begin{proof}
Note that the equivalence classes for $\sim_F$ are unions of blocks for the category
$\mathcal{O}_{\hat{c}}(W)$. Indeed, if $\tau\not\sim_F\xi$, then $\hat{c}_\tau-\hat{c}_\xi\not\in \mathbb{Q}$. Moreover, if $\tau\sim_F\xi'$, then $c^0_{\xi}=c^0_\tau$ for all $c^0\in F^0$.
It follows that $\hat{c}_\xi-\hat{c}_\tau=c_\xi-c_\tau$ and hence is independent
of the choice of $\hat{c}$.

We will produce a module $M'_{\param'}$ in $\OCat_{\param'}(W)$ with the following properties:
the specialization $M'_{\hat{c}}$ coincides with $L_{\hat{c}}(\tau)$, while the specialization
at $c$ has the class in $K_0$ of the form predicted by the statement of the lemma.
This  $M'_{\param'}$ is obtained as a quotient of $\Delta_{\param'}(\tau)$ in several steps
to be described in the next paragraph.

Let $M_{\param'}$ be a module in $\OCat_{\param'}(W)$ and pick $\Xi\subset \Irr(W)$.
Suppose that $M_{\param'}$ is equipped with a $\Z_{\geqslant 0}$-grading $M_{\param'}=\bigoplus_i M_{\param'}(i)$
that is compatible with the grading on $H_{\param'}$.
Note that $M_{\param'}(i)$
is a finitely generated $\C[\param']$-module for each $i$. For each $\xi\in \Xi$, pick
$d_{\xi}\in \Z_{\geqslant 0}$.
The $\C[\param']$-module $\Hom_W(\xi, M_{\param'}(d_\xi)^\h)$ (the superscript indicates
all elements annihilated by $\h$) is finitely generated and hence is generically free. It follows that
$\dim \Hom_W(\xi, M_{\tilde{c}}(d_\xi)^\h)$ is generically constant for $\tilde{c}\in \param'$.
Also we can consider the natural homomorphism
\begin{equation}\label{eq:natur_homom}\bigoplus_{\xi\in \Xi}\Hom_W(\xi, M_{\param'}(d_\xi)^\h)\otimes_{\C[\param']}\Delta_{\param'}(\xi)
\rightarrow M_{\param'}.\end{equation}
This homomorphism preserves gradings if we put $\Hom_W(\xi, M_{\param'}(d_\xi)^\h)$
in degree $d_\xi$ and grade $\Delta_{\param'}(\xi)$ by assigning degree $0$
to $\C[\param']\otimes \xi$.   The  cokernel of (\ref{eq:natur_homom}), denote it by $M^1_{\param'}$, still lies in
$\OCat_{\param'}(W)$ and inherits a grading from $M_{\param'}$. So we can replace
$M_{\param'}$ with $M^1_{\param'}$.

Now let $\tau\in \Irr(W)$. We consider $\Xi=\{\xi\in \Irr(W)\mid \xi\sim_F \tau, \xi<_{\hat{c}}\tau\},
d_\xi=\hat{c}_\xi-\hat{c}_\tau$. Note that $d_\xi$ is independent of the choice of $\hat{c}\in \param'$.
Set $M_{\param'}:=\Delta_{\param'}(\tau)$.  For any quotient $\underline{M}$ of $M_{\param'}$,
a copy of $\xi$ annihilated by $\h$ must appear in degree $d_\xi$,
which is our reason to choose this value of $d_\xi$.
So we have an isomorphism $\Hom_W(\xi, \underline{M}(d_\xi)^\h)\xrightarrow{\sim}
\Hom_{H_{\param'}}(\Delta_{\param'}(\xi),\underline{M})$.
We  apply the construction in the previous paragraph
to $\Xi$ and $M_{\param'}$.  We get the quotient $M^1_{\param'}$ of $M_{\param'}$. We then apply the construction again
but now to $M^1_{\param'}$, getting its quotient $M^2_{\param'}$. The quotients will be proper
as long as $M^k_{\hat{c}}$ is not simple. We stop when $M^k_{\hat{c}}$
is simple, equivalently, $(M^k_{\hat{c}})^{\h}=\tau$.
The object $M^k_c$ satisfies $\Hom_{H_c}(\Delta_c(\xi), M^k_c)=0$ for any $\xi\in \Xi$.
Moreover $[M^k_c]=[M^k_{\hat{c}}]$ by the construction, and $M^k_c\twoheadrightarrow L_c(\tau)$.

Let us show that the kernel of $M^k_c\twoheadrightarrow L_c(\tau)$ is filtered with subquotients $L_c(\eta)$,
where $\eta<_F \tau$. Consider the Serre subcategory $\OCat_{c,<\tau}(W)$ spanned by the simples
$L_c(\tau'), \tau'<_c\tau$. By our choice of $c$ in the beginning
of the section, we have that $\Xi$ is a poset coideal in $\{\tau'\mid\tau'<_c\tau\}$.
The category $\OCat_{c,<\tau}(W)$ is a highest weight subcategory of $\OCat_c(W)$.
Therefore the projective objects $P_{c,<\tau}(\xi),\xi\in \Xi,$ in the category
$\OCat_{c,<\tau}(W)$ are filtered with $\Delta_c(\xi'),\xi'\in \Xi$.
Since $\Hom_{\OCat_{c,<\tau}(W)}(\Delta_c(\xi),M^k_c)=0$ for all $\xi\in \Xi$, we deduce that
$$\Hom_{\OCat_{c,<\tau}(W)}(P_{c,<\tau}(\xi),M^k_c)=0, \quad \forall \xi\in \Xi.$$
It follows that, for any simple composition factor $L_c(\tau')$ of $M_c^k$ different from
$\tau$, we have
$\tau'<_F\tau$. This completes the proof of the lemma.
\end{proof}

\subsubsection{Completion of the proof}\label{SSS_wc_proof_compl}
Now we are ready to prove Proposition \ref{Prop:wc_bij_comp}.

\begin{proof}[Proof of Proposition \ref{Prop:wc_bij_comp}]
Let us write $\OCat_{c, \leqslant_F \tau}(W)$ for the Serre subcategory
of $\OCat_c(W)$ spanned by the simples $L_c(\xi)$ with $\xi\leqslant_F \tau$.
We set $\OCat_{c,\sim_F\tau}(W):=\OCat_{c, \leqslant_F \tau}(W)/\OCat_{c, <_F \tau}(W)$.

Recall that $D^b(\OCat_{c, \leqslant_F \tau}(W))$  is a full subcategory
of $D^b(\OCat_c(W))$. Since $\WC_{c-\psi\leftarrow c}$ maps
$\Delta_c(\tau')$ to an object with class $[\nabla_{c-\psi}(\tau')]$,
it restricts to an equivalence $D^b(\OCat_{c, \leqslant_F \tau}(W))\xrightarrow{\sim}
D^b(\OCat_{c-\psi, \leqslant_F \tau}(W))$. This restriction  satisfies
(P1)-(P3) with respect to the filtrations on  $\OCat_{c, \leqslant_F \tau}(W),
\OCat_{c-\psi, \leqslant_F \tau}(W)$ restricted from the filtrations on $\OCat_{c}(W),
\OCat_{c-\psi}(W)$ making $\WC_{c-\psi\leftarrow c}$ perverse.
Then we have an  equivalence $$D^b(\OCat_{c, \sim_F \tau}(W))\xrightarrow{\sim}
D^b(\OCat_{c-\psi, \sim_F \tau}(W))$$ induced by $\WC_{c-\psi\leftarrow c}$.
This equivalence is the identity on the level of $K_0$
By Lemma \ref{Lem:class_simple_degen}, the classes of
$[L_c(\tau')],[L_{c-\psi}(\tau')]$ in these $K_0$'s are independent of $c$
(of course, as long as $c$ is Zariski generic). And the equivalence satisfies
(P1)-(P3) with respect to filtrations induced from $\OCat_{c, \leqslant_F \tau}(W),
\OCat_{c-\psi, \leqslant_F \tau}(W)$.

On the other hand, an easy induction together with $[\WC_{c-\psi\leftarrow c}]=\operatorname{id}$
shows that $K_0(\Cat^c_i)=K_0(\Cat^{c-\psi}_i)$ (an equality of subgroups
in $K_0(W\operatorname{-mod})$). It follows that,
for the filtration subquotients $\Cat^c_{\sim_F\tau, i},\Cat^{c-\psi}_{\sim_F\tau,i}$
of the categories $\OCat_{c,\sim_F\tau}(W),\OCat_{c-\psi,\sim_F\tau}(W)$,
we also have $K_0(\Cat^c_{\sim_F\tau, i})=K_0(\Cat^{c-\psi}_{\sim_F\tau,i})$.
Since the classes $[L_c(\tau')],[L_{c-\psi}(\tau')]$ are independent of $c$,
the equality $K_0(\Cat^c_{\sim_F\tau,i})=K_0(\Cat^{c-\psi}_{\sim_F\tau,i})$
implies that the labels of simples in $\Cat^c_{\sim_F\tau,i},\Cat^{c-\psi}_{\sim_F\tau,i}$
do not depend on $c$ as long as $c$ is Zariski generic. Since
$[\WC_{c-\psi\leftarrow c}]$ is independent of $c$ as well (it is always the identity),
we deduce that $\wc_{c-\psi\leftarrow c}$ is independent of $c$ too.
\end{proof}

\begin{Rem}
Let us discuss a connection between the filtration making $\WC_{c-\psi\leftarrow c}$
perverse and the filtration by dimension of support. Thanks to Lemma \ref{Lem:wc_Ringel}, on the set of labels, the former
for parameter  $c$ coincides with the latter but for parameter $\hat{c}$ in the notation
of the proof above. Also it may happen that two simples in $\OCat_c$ with the same
supports lie in different filtration pieces for $\WC_{c-\psi\leftarrow c}$, we will
see an example in Section \ref{SS_comput_example}.
\end{Rem}

\subsection{Wall-crossing vs induction and restriction}\label{SS_WC_vs_Res}
Here we are going to prove that the restriction functors intertwine
the wall-crossing functors.

\subsubsection{Shift bimodules and restriction}
Recall the shift bimodule $\B_{\param',\psi}$ from \ref{SSS_WC_constr}.
Our goal  is to understand the bimodule $(\B_{\param',\psi})_{\dagger,\underline{W}}$.
Consider the analog of $\B_{\param',\psi}$ for $\underline{W}$, the
$H_{\param'-\psi}(\underline{W})$-$H_{\param'}(\underline{W})$-bimodule
$\underline{\B}_{\param',\psi}$.

\begin{Lem}\label{Lem:shift_bimod_restr}
We have an isomorphism $(\B_{\param,\psi})_{\dagger,\underline{W}}\cong
\underline{\B}_{\param,\psi}$ of $H_{\param-\psi}(\underline{W})$-$H_{\param}(\underline{W})$-bimodules.
\end{Lem}
\begin{proof}
The proof is in several steps.

{\it Step 1}. Recall that we have isomorphisms $\iota: e_{\chi}H_{\param}e_\chi\xrightarrow{\sim}
e H_{\param+\bar{\chi}}e$ and $\underline{\iota}:\underline{e}_{\chi}H_{\param}(\underline{W})\underline{e}_\chi\xrightarrow{\sim}
\underline{e} H_{\param+\bar{\chi}}(\underline{W})\underline{e}$, see Lemma \ref{Lem:spher_iso}. Our goal is to relate
these isomorphisms. For this, we will first produce an isomorphism
$\underline{\iota}':\underline{e}_{\chi}H_{\param}(\underline{W})\underline{e}_\chi\xrightarrow{\sim}
\underline{e} H_{\param+\bar{\chi}}(\underline{W})\underline{e}$ from $\iota$.

The isomorphism $\iota$ induces an isomorphism
of completions $\iota^{\wedge_b}: e_{\chi}H_{\param}^{\wedge_b}e_\chi \cong e H^{\wedge_b}_{\param+\bar{\chi}}e$.
Using the isomorphism $\theta_b$ from Proposition \ref{Prop:compl_iso1},
we transfer $\iota^{\wedge_b}$ to an isomorphism $$\underline{e}_{\chi}H_{\param}^{\wedge_b}(\underline{W},\h)\underline{e}_\chi \cong
\underline{e} H^{\wedge_b}_{\param+\bar{\chi}}(\underline{W},\h)\underline{e}$$
to be denoted again  by $\iota^{\wedge_b}$. This isomorphism preserves the filtrations
and is the identity on the associated graded modulo $\param^*$.

We have inclusions $D(\h^{\underline{W}})^{\wedge_b}\hookrightarrow \underline{e}_{\chi}H_{\param}^{\wedge_b}(\underline{W},\h)\underline{e}_\chi,
\underline{e} H^{\wedge_b}_{\param+\bar{\chi}}(\underline{W},\h)\underline{e}$ (here $D(\bullet)$
stands for the algebra of differential operators).
These embeddings are strictly compatible with filtrations
and, after passing to the associated graded algebras, coincide with $\theta^0:\C[\h^{\underline{W}}\oplus
\h^{\underline{W}*}]^{\wedge_b}\hookrightarrow (\C[\h\oplus \h^*]^W)^{\wedge_b}$. One can show
that any two such embeddings differ by the conjugation with $\exp(\operatorname{ad} f)$,
where $f\in \C[\h/\underline{W}]^{\wedge_b}\otimes (\C\oplus \param^*)$, compare
to Section \ref{SS_compl}.
We can change $\theta_b$ and assume that $\iota^{\wedge_b}$ intertwines the embeddings
of  $D(\h^{\underline{W}})^{\wedge_b}$. So $\iota^{\wedge_b}$ restricts to an isomorphism
of the centralizers of $D(\h^{\underline{W}})^{\wedge_b}$, i.e., to
$\underline{\iota}':\underline{e}_{\chi}H_{\param}^{\wedge_0}(\underline{W})\underline{e}_\chi
\xrightarrow{\sim} \underline{e} H^{\wedge_0}_{\param+\bar{\chi}}(\underline{W})\underline{e}$.

Let $\underline{h}_\chi\in \underline{e}_{\chi}H_{\param}^{\wedge_0}(\underline{W})\underline{e}_\chi,
\underline{h}\in \underline{e} H^{\wedge_0}_{\param+\bar{\chi}}(\underline{W})\underline{e}$ denote the Euler elements.
Since the $\underline{\iota}^{\wedge_0}$ preserves the filtrations and is the identity
modulo $\param^*$ on the associated graded algebras, we see that
$\underline{\iota}'(\underline{h}_\chi)=\underline{h}+F$, where
$F\in \C[\underline{\h}/\underline{W}]^{\wedge_0}\otimes (\C\oplus \param^*)$.
From here and the fact that the eigenvalues of $\underline{h}$ on the augmentation ideal in
 $\C[\underline{\h}/\underline{W}]$ are positive we deduce that there is $F'\in \C[\underline{\h}/\underline{W}]^{\wedge_0}\otimes
(\C\oplus \param^*)$ such that $\underline{\iota}'(\underline{h}_\chi)-
\exp(\operatorname{ad}F')\underline{h}\in \C\oplus \param^*$. Twisting $\theta_b$ by
$\exp(\operatorname{ad}F')$ (where we view $\C[\underline{\h}/\underline{W}]^{\wedge_0}
$ as a subalgebra of $\C[\h/W]^{\wedge_b}$ using $(\theta^0)^{-1}$) we achieve that
$\underline{\iota}'$ intertwines $\operatorname{ad}(\underline{h}),
\operatorname{ad}(\underline{h}_\chi)$.  Since $\underline{e} H_{\param+\bar{\chi}}(\underline{W})\underline{e}$
is the space of $\operatorname{ad}(\underline{h})$-finite elements in
$\underline{e} H^{\wedge_0}_{\param+\bar{\chi}}(\underline{W})\underline{e}$
(and the similar claim holds for $\underline{e} H_{\param}(\underline{W})\underline{e}$),
we see that $\underline{\iota}'$ restricts to $
\underline{e}_\chi H_{\param}(\underline{W})\underline{e}_\chi\xrightarrow{\sim}\underline{e} H_{\param+\bar{\chi}}(\underline{W})\underline{e}$.

{\it Step 2}. Now we have two isomorphisms $\underline{\iota}, \underline{\iota}':\underline{e}_\chi H_{\param}(\underline{W})\underline{e}_\chi\xrightarrow{\sim}\underline{e} H_{\param+\bar{\chi}}(\underline{W})\underline{e}$.
Both preserve the filtrations and give the identity on the associated graded algebras
modulo $\param^*$. From  here we deduce that there is an element $f\in \C[\underline{\h}/\underline{W}]\otimes
(\C\oplus \param^*)$ such that $\underline{\iota}=\underline{\iota}'\circ \operatorname{ad}(f)$.
Further modifying $\theta_b$, we may assume that $\underline{\iota}=\underline{\iota}'$.


{\it Step 3}. It follows from Step 2 and the construction of the restriction functor
that $(e_\chi H_{\param'}e)_{\dagger,\underline{W}}\cong
\underline{e}_\chi H_{\param'}(\underline{W})\underline{e}$, an isomorphism
of $\underline{e}_\chi H_{\param'}(\underline{W})\underline{e}_\chi$-$\underline{e} H_{\param'}(\underline{W})\underline{e}$-bimodules.
From Step 1 we deduce that    $(e_\chi H_{\param'}e)_{\dagger,\underline{W}}\cong
\underline{e}_\chi H_{\param'}(\underline{W})\underline{e}$, an isomorphism of
$\underline{e} H_{\param'+\bar{\chi}}(\underline{W})\underline{e}$-$\underline{e} H_{\param'}(\underline{W})
\underline{e}$-bimodules.

{\it Step 4}. Since the functor $\bullet_{\dagger,\underline{W}}$ intertwines
all tensor products involved in the construction of $\B_{\param',\psi}$
we deduce that $(\B_{\param',\psi})_{\dagger,\underline{W}}\cong \underline{\B}_{\param',\psi}$.
\end{proof}

\subsubsection{Wall-crossing functors and restriction}
\begin{Prop}\label{Prop:WC_restr}
Suppose that $\hat{c},\hat{c}-\psi$ lie in opposite chambers for both $\underline{W}$
and $W$ provided $\hat{c}$ is Weil generic in $\param'$. Then, for a Zariski generic
$c\in \param'$, we have a natural isomorphism of functors $$\underline{\WC}_{c-\psi\leftarrow c}\circ
\Res^W_{\underline{W}}\cong \Res^W_{\underline{W}}\circ \WC_{c-\psi\leftarrow c}.$$
Here we write $\underline{\WC}_{c-\psi\leftarrow c}$ for the wall-crossing functor
$D^b(\OCat_c(\underline{W}))\rightarrow D^b(\OCat_{c-\psi}(\underline{W}))$.
\end{Prop}
\begin{proof}
It follows from Lemma \ref{Lem:shift_bimod_restr} that
$(\B_{c,\psi})_{\dagger,\underline{W}}=\underline{\B}_{c,\psi}$.
By \cite[Proposition 4.11]{rouq_der}, we get $\underline{\B}_{\hat{c}}(\psi)\cong
(\B_{\hat{c}}(\psi))_{\dagger,\underline{W}}$. So we can set
$\underline{\B}_{\param'}(\psi)=(\B^2/\B^1)_{\dagger,\underline{W}}=
(\B^2)_{\dagger,\underline{W}}/(\B^1)_{\dagger,\underline{W}}$,
where $\B^1,\B^2$ are as in  \ref{SSS_WC_constr}.
This shows $\underline{\B}_{c}(\psi)=(\B_c(\psi))_{\dagger,\underline{W}}$.
Using 2) from \ref{SSS_res_HC_prop}, we complete the proof.
\end{proof}

The next corollary follows from Proposition \ref{Prop:WC_restr} and the adjointness
properties of $\Res^W_{\underline{W}}$ and $\Ind^W_{\underline{W}}$.

\begin{Cor}\label{Cor:WC_restr}
Under the assumptions of Proposition \ref{Prop:WC_restr},
we have $$\WC_{c-\psi\leftarrow c}\circ
\Ind^W_{\underline{W}}\cong \Ind^W_{\underline{W}}\circ \underline{\WC}_{c-\psi\leftarrow c}.$$
\end{Cor}

\subsubsection{Assumptions on $\hat{c}-\psi,\hat{c}$}\label{SSS_essent_walls}
We would like to make some comments on the assumption in Proposition \ref{Prop:WC_restr}
that $\hat{c}-\psi,\hat{c}$ are opposite for $\underline{W}$. We basically need to impose this
assumption because the decomposition into chambers that we have chosen (according to
the c-function) is too fine. We can choose a rougher decomposition, for example, as follows.
Recall that we have fixed a coset $c+\underline{\param}_{\Z}$. Let $\Pi_{\tau,\xi}$
be a hyperplane as in Section \ref{SS_chamber}. We say that $\Pi_{\tau,\xi}$ is
{\it essential} (for $c+\underline{\param}_{\Z}$) if, for a Weil generic $\hat{c}\in c+\Pi_{\tau,\xi}$, the category
$\mathcal{O}_{\hat{c}}$ is not semisimple.

\begin{Lem}\label{Lem:non_essent_cross}
Let $c,c-\psi$ lie in chambers that share a wall spanning a non-essential hyperplane.
Then $\WC_{c-\psi\leftarrow c}$ is a highest weight equivalence.
\end{Lem}
\begin{proof}
First of all, let us observe that if the category $\OCat_{\hat{c}}(W)$ is semisimple,  then the algebra
$H_{\hat{c}}$ is simple. If $\OCat_{\hat{c}}(W)$ is semisimple, then every Verma module is simple, hence the annihilator
of every simple module in $\mathcal{O}_{\hat{c}}(W)$ is zero.  Every
prime ideal in $H_{\hat{c}}$ is primitive, while every primitive ideal is
the annihilator of a simple module from $\OCat_{\hat{c}}(W)$, see \cite{Ginzburg_prim}.
We conclude that the only prime ideal in $H_{\hat{c}}$ is zero. It follows that $H_{\hat{c}}$
contains no nontrivial two-sided ideals.

It follows that $I_{n-1}^{\hat{c}}=H_{\hat{c}}$
and $I_{n-1}^{\hat{c}-\psi}=H_{\hat{c}}$. So $I_{n-1}^c=H_c$ and $I_{n-1}^{c-\psi}=H_{c-\psi}$.
Therefore $\WC_{c-\psi\leftarrow c}$ is an equivalence of abelian categories.

It remains to check that $\WC_{c-\psi\leftarrow c}\Delta_c(\tau)=\Delta_{c-\psi}(\tau)$ for all
$\tau$. The functor $\WC_{c-\psi\leftarrow c}$ is the identity on $K_0$ by Proposition \ref{Prop:WC_K0}.
So it maps the projective Vermas in $\OCat_{c}(W)$ to the projective Vermas in  $\OCat_{c-\psi}(W)$
and those have the same labels in the two categories. Then we restrict  $\WC_{c-\psi\leftarrow c}$
to the Serre spans of the remaining Vermas and repeat the argument.
\end{proof}

So we can partition $c+\underline{\param}_{\Z}$ into chambers using essential
hyperplanes only. It is still unclear whether the condition that $\hat{c},\hat{c}-\psi$
lie in opposite chambers for $W$ implies that for $\underline{W}$, in general.
However, this is true when we consider chambers separated by a single
essential hyperplane.

\begin{Prop}\label{Prop:WC_restr_essent}
Let $c\in \param$, $\Pi$ be an essential wall for $c+\underline{\param}_{\Z}$ and let $\psi\in
\underline{\param}_{\Z}$ be such that $c,c-\psi$ are separated by $\Pi$ and are not separated
by any other essential wall. Then, after replacing
$c$ with a Zariski generic element of $c+(\Pi\cap \underline{\param}_{\Z})$, we have that either
$\Pi$ is an essential wall for $\underline{W}$ or $c,c-\psi$ are not separated by any essential
wall for $\underline{W}$. Moreover, we have a natural isomorphism of functors
$$\underline{\WC}_{c-\psi\leftarrow c}\circ \Res^W_{\underline{W}}\cong \Res^W_{\underline{W}}\circ \WC_{c-\psi\leftarrow c}.$$
\end{Prop}
\begin{proof}
If $\Pi$ is an essential wall for $\underline{W}$, then our claim follows from
Proposition \ref{Prop:WC_restr}. Assume that $\Pi$ is not an essential wall
for $\underline{W}$. From the proof of Lemma \ref{Lem:non_essent_cross} it follows
that $\underline{\B}_c(\psi)=\underline{\B}_{c,\psi}$. Also it follows
that the bimodules $\B^1_{\dagger,\underline{W}},(\B_{c+\Pi,\psi}/\B^2)_{\dagger,\underline{W}}$
are torsion over $\C[\Pi]$. Therefore $\B_c(\psi)_{\dagger,\underline{W}}=\underline{\B}_c(\psi)$.
The required isomorphism of functors follows from 2) of \ref{SSS_res_HC_prop}.
\end{proof}

\begin{Rem}\label{Rem:walls_separation}
Let us explain how to determine which essential walls separate two parameters
$c,c'$ with $c'-c\in \underline{\param}_{\Z}$. Let $\Pi_1,\ldots,\Pi_k$ be all
essential walls for $c$ (equivalently, for $c'$). Let $\varphi_1,\ldots\varphi_k\in
\param_{\Q}^*$ be functionals defining $\Pi_1,\ldots,\Pi_k$. Then $c,c'$ are
separated exactly by the walls $\Pi_i$ such that $\varphi_i(c),\varphi_i(c')$
have different signs.
\end{Rem}

In any case, suppose that $\hat{c},\hat{c}-\psi$ lie in opposite chambers for
$W$. Then one can show that $\OCat_{\hat{c}-\psi}(\underline{W})\cong \OCat_{\hat{c}}(\underline{W})^\vee$
and the functor $$\underline{\B}_{\hat{c}-\psi\leftarrow \hat{c}}\otimes^L_{H_{\hat{c}}(\underline{W})}\bullet$$
becomes the inverse Ringel duality under this identification. In particular,
this functor becomes $\underline{\WC}_{\hat{c}-\underline{\psi},\hat{c}}$ up to pre-composing
with a highest weight equivalence, where $\underline{\psi}$
is such that $\hat{c}, \hat{c}-\underline{\psi}$ lie in opposite chambers for
$\underline{W}$. We do not prove this claim as we do not need it.

\section{Cyclotomic categories $\mathcal{O}$}\label{S_cyclot}
\subsection{Cyclotomic Cherednik algebras and their categories $\mathcal{O}$}\label{SS_cyclot_O}
We now concentrate on the case when $W=G(\ell,1,n)$. Here we have either one
(if $\ell=1$ or $n=1$) or two conjugacy classes of reflection hyperplanes.
Let $H^1$ denote the hyperplane $x_1=0$ and $H^2$ denote the hyperplane
$x_1=x_2$. We set $\kappa=-c(s)$, where $s$ is a transposition in $\mathfrak{S}_n\subset W$.
If $\kappa=0$, then $H_{c}=H_c(1)^{\otimes n}\#S_n$, where we write $H_c(1)$
for the rational Cherednik algebra corresponding to $(\Z/\ell\Z,\C)$. This is a very easy
case and we are not going to consider it here (we explain how to compute
the supports in Section \ref{S_App_kappa0}). So we assume that $\kappa\neq 0$.
Define complex numbers $s_i, i=1,\ldots,\ell,$ by $h_{H^1,i}=\kappa s_i-\frac{i}{\ell}$
(below we will write $h_i$ for $h_{H^1,i}$).
We view $(s_1,\ldots,s_\ell)$ as an $\ell$-tuple of complex numbers defined
up to a common summand.

Below we will  write $\OCat_c(n)$ for the category $\OCat_c(G(\ell,1,n))$.
We will write $\OCat_c$ for the direct sum $\bigoplus_{n\geqslant 0}\OCat_c(n)$.

\subsubsection{$\Irr(W)$ vs multipartitions}
Let us proceed to the combinatorial description of the set $\Irr(W)$.
We can identify this set with the set $\mathcal{P}_\ell(n)$ of $\ell$-multipartitions
of $n$ as follows. To a multipartition $\lambda$, we assign the irreducible
$\prod_{i=1}^\ell G(|\lambda^{(i)}|,1,\ell)$-module $\bigotimes_{i=1}^\ell S_{\lambda^{(i)}}$,
where an element $\eta$ in any copy of $\Z/\ell \Z$ inside $G(|\lambda^{(i)}|,1,\ell)$
acts on  the irreducible $\Sym_{|\lambda^{(i)}|}$-module $S_{\lambda^{(i)}}$ (corresponding to the diagram $\lambda^{(i)}$)
by $\exp(2\pi\sqrt{-1}i\eta/\ell)$. The irreducible $G(\ell,1,n)$-module corresponding
to $\lambda$ is obtained from $\bigotimes_{i=1}^\ell S_{\lambda^{(i)}}$
by induction.

\subsubsection{Blocks and order on boxes}\label{SSS_blocks_order}
Now we are going to describe the blocks of $\OCat_c$. To a box $b$ with column number $x$,
row number $y$ and diagram number $i$ we assign its shifted content $\cont(b):=x-y+s_i$.
The following is a reformulation of \cite[Lemma 5.16]{SV}.

\begin{Lem}\label{Lem:blocks}
Two multipartitions $\lambda$ and $\lambda'$ lie in the same block if and only if  the multisets
$\{\cont(b) \mod \kappa^{-1}\Z, b\in\lambda\}$ and $\{\cont(b)\mod \kappa^{-1}\Z,
b\in \lambda'\}$  are the same.
\end{Lem}

Now let $b=(x,y,i),b'=(x',y',i')$ be two boxes such that their contents modulo $\kappa^{-1}\Z$
are the same (we say that such boxes are equivalent).
When $b\sim b'$, we say that $b\leqslant b'$ if $\kappa \ell \cont(b)-i\geqslant
\kappa \ell \cont(b')-i'$.

\subsubsection{Lattices $\param_{\Z}$ and $\underline{\param}_{\Z}$}
Note that for $c=(\kappa,s_1,\ldots,s_\ell), c'=(\kappa',s_1',\ldots,s_\ell')$,
the inclusion $c-c'\in \param_{\Z}$ means that $\kappa-\kappa'\in \Z$
and $(h_i-h_j)-(h_i'-h_j')=\kappa (s_i-s_j)-\kappa' (s_i'-s_j')\in \Z$. The sublattice
$\underline{\param}_{\Z}$ coincides with $\param_{\Z}$ in this case
because, for the groups $G(\ell,1,n)$, the KZ twist is trivial,
see \cite[6.4.7]{GL} for explanation and references.

\subsubsection{Switching $\kappa$ to $-\kappa$}
Let $\chi$ be the character of $G(\ell,1,n)$ that is equal to the sign on $\Sym_n$
and is the identity on $\Z/\ell\Z$. Then, if $c\in \param, c=(\kappa,s_1,\ldots,s_\ell)$,
we have $c^\chi=(-\kappa, -s_1,\ldots,-s_\ell)$. The corresponding equivalence
$\OCat_c\xrightarrow{\sim}\OCat_{c^\chi}$ maps $\Delta_c(\lambda)$ to $\Delta_{c^\chi}(\lambda^t)$,
where $\lambda^t$ denotes the componentwise transpose of the $\ell$-multipartition
$\lambda$.

\subsection{Categorical actions}\label{SS_cat_act}
In this section we consider categorical actions of Lie algebras on $\OCat_c$
and some related structures.

\subsubsection{Categorical Kac-Moody action}\label{SSS_cat_KM}
There is a categorical type A Kac-Moody action (in the sense of \cite[5.3.7,5.3.8]{Rouquier_2Kac})
on $\mathcal{O}_c$ defined in \cite[Section 5]{Shan}. Such an action consists of biadjoint
functors $E,F$ (with fixed one sided adjunction morphisms) and functor endomorphisms
$X\in \End(E)$ and $T\in \End(E^2)$. We are going to recall the functors $E$ and $F$
and also how they split into the eigen-functors for $X$, but we are not going to
recall the other parts of the data.

We set $F:=\bigoplus_{n\geqslant 0}\Ind_n^{n+1}$ and $E:=\bigoplus_{n\geqslant 0}\Res^n_{n-1}$,
where we write $\Ind_n^{n+1}$ for the induction functor from $G(\ell,1,n)$ to $G(\ell,1,n+1)$
and $\Res^n_{n-1}$ for $n>0$ has the similar meaning, while $\Res^0_{-1}=0$.

Let $z\in \C/\kappa^{-1}\Z$. We say that a box $b=(x,y,i)$ is a $z$-box if $x-y+s_i$
is congruent to $z$ modulo $\kappa^{-1}\Z$.  For a module $M$ in a block corresponding
to a multiset $A$, we define $F_z M$ as the projection of $FM$ to the block corresponding
to $A\cup \{z\}$. We define $E_z M$ as the projection of $EM$ to the block corresponding
to $A\setminus \{z\}$. Then $(F_z,E_z)$ define a highest weight categorical $\slf_2$-action
on $\OCat_c$, see \cite[Section 4.2]{cryst}. In particular, $E_z^n$ decomposes
as the direct sum of $n!$ copies of the divided power functor $E_z^{(n)}$,
see \cite[5.2.2]{CR}. Similarly, we get the functor $F_z^{(n)}$.

The functors $E_z,F_z$ give rise to a categorical action of a Kac-Moody algebra $\g_c$
that is determined as follows. For $i,j\in \{1,\ldots,\ell\}$, we write $i\sim_c j$
if $s_i-s_j\in \Z+\kappa^{-1}\Z$ (i.e., the diagrams number $i$ and $j$ can have equivalent
boxes). Then the algebra
acting is the product of several copies of $\hat{\slf}_e$ (where $e$ is the denominator
of $\kappa$), one copy per each equivalence class for $\sim_c$. The complexified $K_0$
is the product $\bigotimes_\alpha \mathcal{F}_\alpha$, where $\alpha$ runs
over the equivalence classes for $\sim_c$ and we write $\mathcal{F}_\alpha$
for the level $|\alpha|$ Fock space whose basis is indexed by multipartitions
of the form $(\lambda^{(i)})$, where $\lambda^{(i)}=\varnothing$ if $i\not\in\alpha$.

\subsubsection{Kac-Moody crystal}\label{SSS_crystal}
For a category $\Cat$ equipped with a categorical  action of a Kac-Moody
algebra $\g$, the set  $\Irr(\Cat)$ comes equipped with  the structure
of a $\g$-crystal. Let us recall the construction. Pick a simple $L$. It was checked in \cite[Proposition 5.20]{CR}
(for any categorical $\slf_2$-action with functors $E_z,F_z$)
that, whenever the object $E_zL\neq 0$, it has isomorphic head and socle that are simple.
The same is true for $F_z L$. The  operator $\tilde{e}_z$
sends $L$ to the socle of $E_z L$  provided it is nonzero and to zero otherwise.
The operator $\tilde{f}_z$ is defined in a similar fashion. Below we often
use the identification of $\mathcal{P}_\ell$ with the set of irreducibles in
$\OCat_c$ to carry the crystal structure to $\mathcal{P}_\ell$.

Now let us recall how to compute the crystal for the $\g_c$-action on
$\OCat_c$, see \cite[Theorem 5.1]{cryst}. To compute $\tilde{f}_z L_c(\lambda)$ and $\tilde{e}_z L_c(\lambda)$
consider all addable and removable $z$-boxes in $\lambda$. We place them in the decreasing order and write $+$ for an addable box and $-$ for a removable one (we call this collection the {\it $z$-signature} of $\lambda$). Then we consequently remove all instances of $-+$ getting what we call the {\it reduced signature}.
The operator $\tilde{f}_z$ adds the box that corresponds to the rightmost remaining $+$, the operator $\tilde{e}_z$ removes the box corresponding to the leftmost remaining $-$. If the reduced signature consists of all $+$'s (resp, all $-$'s),
then $\tilde{e}_z \lambda=0$ (resp., $\tilde{f}_z \lambda=0$).

\subsubsection{Heisenberg action}\label{SSS_Heisenberg}
Now suppose that $\kappa<0$ is a rational number, while  $e\kappa s_1,\ldots,e\kappa s_\ell\in \Z$. Then we get
functors $E_i,F_i$, one per residue mod $e$. The based space $(K_0^{\C}(\OCat_c), [\Delta_c(\tau)], \tau\in
\Part_\ell)$ becomes the level $\ell$ Fock space with multicharge $(s_1,\ldots, s_\ell)$
(note that  the Fock space makes sense as an $\hat{\slf}_e$-module as long as $s_1,\ldots,s_\ell$
are rational numbers whose denominators are coprime to $e$).

We can realize $\underline{W}:=G(\ell,1,n-m)\times \mathfrak{S}_m$ as a parabolic subgroup of $W:=G(\ell,1,n)$.
We will need some functors obtained from the induction functors from $G(\ell,1,n-em)\times \mathfrak{S}_{em}$
to $G(\ell,1,n)$. Let $\mu$ be a partition of $m$. Consider the functor $A_\mu:=\operatorname{Ind}_W^{\underline{W}}\left(\bullet\boxtimes L^A(e\mu)\right)$,
where $L^A(e\mu)$ stands for the simple in $\mathcal{O}_\kappa(\mathfrak{S}_{em})$ indexed by the partition
$e\mu:=(e\mu_1,e\mu_2,\ldots)$ of $em$. As Shan and Vasserot checked in \cite[Section 5.3]{SV}, the functors $A_\mu$
commute with $E_i,F_i$ for all $i$ and $\mu$. On the level of $K_0$, the functors $A_\mu$ and their
derived right adjoint functors $RA_\mu^*$ give rise to a Heisenberg action. 

We note that in \cite{SV} there was an assumption that $e>2$. It was needed to make sure that
the category $\mathcal{O}_\kappa(\mathfrak{S}_n)$ is equivalent to the category of modules
over the $q$-Schur algebra $\mathcal{S}_q(n)$, where $q=\exp(\pi\sqrt{-1}\kappa)$. This is trivial
when $e=1$ (both categories are isomorphic to $\mathfrak{S}_n\operatorname{-mod}$) and was
established in \cite[Appendix]{VV_proof} when $e=2$.

If $\kappa>0$, we still have a categorical Heisenberg action: we need to set $A_\mu:=
\operatorname{Ind}^{\underline{W}}_W\left(\bullet\boxtimes L^A((e\mu)^t)\right)$.

In fact, one can remove the condition that  $e\kappa s_1,\ldots,e\kappa s_\ell\in \Z$.
In the general case the functors $A_\mu$ are going to decompose into direct sums and will lead to several
commuting Heisenberg actions. We do not need this construction so we are not going to
elaborate on this.

\subsection{Decomposition}\label{SS_equi}
When $\kappa$ is irrational, the computation
of supports was done in \cite{cryst}. So we assume that the denominator $e$ of
$\kappa$ is finite. Below we will write $\OCat_{\kappa,{\bf s}}$ for $\OCat_c$
(where ${\bf s}=(s_1,\ldots,s_\ell)$).

Suppose that there is more than one equivalence class (with respect to $\sim_c$) of the indexes
$1,\ldots,\ell$. For an equivalence class $\alpha$, set ${\bf s}^\alpha:=(s_i)_{i\in \alpha}$.
Let us write $\OCat_{\kappa,{\bf s}^\alpha}$
for the category $\mathcal{O}$ for $G(|\alpha|,1,n)$, this category comes equipped
with a categorical action of the factor of $\g_c$ corresponding to $\alpha$.
We take the set of all multipartitions that have zero entries outside of $\alpha$ for the labelling set
of $\OCat_{\kappa,{\bf s}^\alpha}$.  So the simples in $\boxtimes_\alpha \OCat_{\kappa,{\bf s}^\alpha}$
are labelled by $\Part_\ell$. Also note that $\boxtimes_\alpha \OCat_{\kappa,{\bf s}^\alpha}$
comes with the tensor product action of $\g_c$.

\begin{Lem}\label{Lem:cat_equi_decomp}
There is an equivalence $\OCat_{\kappa,{\bf s}}\xrightarrow{\sim}\boxtimes_\alpha \OCat_{\kappa,{\bf s}^\alpha}$
mapping $\Delta(\lambda)$ to $\Delta(\lambda)$ and strongly equivariant for the $\g_c$-action.
\end{Lem}

The proof is given in \cite[Section 6.4]{rouqqsch} in the case when $e\neq 2$
and $s_i-s_j$ is not divisible by $e$. This restriction can be removed using techniques
from \cite[Section 4.2]{rouq_der}. The strong equivariance follows from the
construction of the equivalence in {\it loc. cit.} and of the categorical action in \cite{Shan}.
Similarly, as we have remarked in the end of  \ref{SSS_Heisenberg}, one can equip
$\OCat_{\kappa,{\bf s}}$ with several commuting Heisenberg actions, one action
per equivalence class. The equivalence of Lemma \ref{Lem:cat_equi_decomp} becomes
equivariant with respect to the Heisenberg actions.

For $\lambda\in \Part_\ell$, let $\lambda^\alpha$ denote the collection
of components of $\lambda$ in the diagrams of class $\alpha$.

\begin{Cor}\label{Cor:decomp_supp}
We have $p_{\kappa,{\bf s}}(\lambda)=\sum_\alpha p_{\kappa,{\bf s}^\alpha}(\lambda^\alpha)$
and $q_{\kappa,{\bf s}}(\lambda)=\sum_\alpha q_{\kappa,{\bf s}^\alpha}(\lambda^\alpha)$.
\end{Cor}
\begin{proof}
The first equality follows from the strong equivariance of the equivalence in
Lemma \ref{Lem:cat_equi_decomp} combined with the computation of
$p(\lambda)$ from \cite[Section 5.5]{cryst}. Also note that the codimensions of support
of $L(\lambda)$ in $\OCat_{\kappa,{\bf s}}$ and $ \boxtimes_\alpha \OCat_{\kappa,{\bf s}^\alpha}$
coincide. This is because these codimensions are recovered from the highest
weight structure, Lemma \ref{Lem:titl_support}. From here and the equality
$p_{\kappa,{\bf s}}(\lambda)=\sum_\alpha p_{\kappa,{\bf s}^\alpha}(\lambda^\alpha)$
we deduce  $q_{\kappa,{\bf s}}(\lambda)=\sum_\alpha q_{\kappa,{\bf s}^\alpha}(\lambda^\alpha)$.
\end{proof}

\section{Proofs of main results}\label{S_main_proofs}
\subsection{Heisenberg crystal}\label{SS_Heis_cryst}
Let $e$ denote the denominator of $\kappa$. Suppose $\kappa e s_1,\ldots,\kappa e s_\ell$ are all integers.
Here we will define a level 1  $\slf_\infty$-crystal on $\Irr(\OCat_c)=\mathcal{P}_\ell$ such that $q_c(\lambda)$
is the depth in this crystal.  The restriction of this crystal to $\{\lambda\mid p(\lambda)=0\}$ appeared implicitly in \cite[Section 5.6]{SV}.   Shan and Vasserot did not describe this structure as a crystal but they gave a basically equivalent description.

\subsubsection{The case $p_c(\lambda)=0$}\label{SSS_Heis_cryst_p0}
Let us start by establishing an $\slf_\infty$-crystal structure on the set  $\{\lambda\in \mathcal{P}_\ell\mid p_c(\lambda)=0\}$.

Let $\lambda$ be such that $L_c(\lambda)$ is finite dimensional, equivalently, $p_c(\lambda)=q_c(\lambda)=0$.
Then the structure of $A_\mu L_c(\lambda)$ is as follows, see \cite[Sections 5.4-5.6]{SV}:
there is a uniquely determined multipartition
$\tilde{a}_\mu\lambda\in \mathcal{P}_\ell(|\lambda|+e|\mu|)$ with  $L_c(\tilde{a}_\mu \lambda)$ being a subquotient
of $A_\mu L_c(\lambda)$  satisfying $p_c(\tilde{a}_\mu\lambda)=0$ and $q_c(\tilde{a}_\mu \lambda)=|\mu|$. Any other subquotient $L_c(\lambda')$ of $A_\mu L_c(\lambda)$ satisfies
$p_c(\lambda')=0, q_c(\lambda')<|\mu|$. The module $L_c(\lambda')$ cannot occur in the socle or in the head
of $A_\mu L(\lambda)$ by Lemma \ref{Lem:Ind_support}.

Further, it is shown in \cite[Section 5.6]{SV}, that, for any $q\in \Z_{>0}$, the map
\begin{equation}\label{eq:bijection}\mathcal{P}_1(q)\times
\{\lambda\in \mathcal{P}_\ell\mid p_c(\lambda)=q_c(\lambda)=0\}\rightarrow \{\lambda\in
\mathcal{P}_\ell\mid p_c(\lambda)=0, q_c(\lambda)=q\}, (\mu,\lambda)\mapsto \tilde{a}_\mu\lambda\end{equation}
is a bijection. The resulting bijection $$\mathcal{P}_1\times
\{\lambda\in \mathcal{P}_\ell\mid p_c(\lambda)=q_c(\lambda)=0\}\xrightarrow{\sim}\{\lambda\in \mathcal{P}_\ell\mid p_c(\lambda)=0\}$$
produces a  (level 1) $\slf_\infty$-crystal on the target space, carried over from the standard
$\slf_\infty$-crystal on $\mathcal{P}_1$.

By the construction, $q_c(\lambda)$ is the depth of $\lambda$ in the $\slf_{\infty}$-crystal.

\subsubsection{The general case}
Now let us extend the $\slf_\infty$-crystal to the whole set $\Irr(\OCat_c)=\mathcal{P}_\ell$.
A crucial step is the following claim.

\begin{Prop}\label{Prop:cryst_commut_einfty}
Let $\lambda^0$ be a multipartition with $p(\lambda^0)=q(\lambda^0)=0$, let $C_f$ be a composition
of $\tilde{f}_i$'s such that $\lambda:=C_f \lambda^0\neq 0$, and let $\mu$ be a partition. Set
$\tilde{\lambda}^0:=\tilde{a}_\mu \lambda^0$. Then the head of $A_\mu L(\lambda)$
is a nonzero multiple of $L(\tilde{\lambda})$, where $\tilde{\lambda}:=C_f \tilde{\lambda}^0$.
\end{Prop}
\begin{proof}
First of all, since $p_{c}(\tilde{\lambda}^0)=0$, the multipartition $\tilde{\lambda}^0$
is a singular vertex in the $\hat{\slf}_e$-crystal. Since the weights of $\lambda^0, \tilde{\lambda}^0$
coincide, the connected components of the crystal through $\lambda^0, \tilde{\lambda}^0$
are isomorphic (and the isomorphism is unique and  maps $\tilde{\lambda}^0$ to $\tilde{\lambda}$).
In particular, $\tilde{\lambda}\neq 0$.

Now we can prove our claim by the induction on the length of $C_f$. The case when the length is
$0$ is trivial. Now suppose that the claim is proved for all lengths less than some $N$. We are going to prove it
for $C_f$ of length $N$. First of all, we will modify $C_f$ without changing $\lambda,\tilde{\lambda}$
(and therefore preserving the length). Namely, we can find indexes $i_1,\ldots,i_k$ and positive
integers $n_1,\ldots,n_k$ summing to $N$ such that
\begin{itemize}
\item $\tilde{e}_{i_1}^{n_1} \tilde{e}_{i_2}^{n_2}\ldots \tilde{e}_{i_k}^{n_k}\lambda=\lambda^0$.
\item $\tilde{e}_{i_j}^{n_j+1}\tilde{e}_{i_{j+1}}^{n_{j+1}}\ldots \tilde{e}_{i_k}^{n_k}\lambda=0$ for all $j$.
\end{itemize}
We set $C_f:= \tilde{f}_{i_k}^{n_k}\ldots \tilde{f}_{i_1}^{n_1}$, it maps $\lambda^0$ to $\lambda$
and $\tilde{\lambda}^0$ to $\tilde{\lambda}$.

Set $\lambda':=\tilde{e}_{i_k}^{n_k}\lambda, \tilde{\lambda}':=\tilde{e}_{i_k}^{n_k}\tilde{\lambda}$,
we will write $i$ instead of $i_k$ and $n$ instead of $n_k$ to simplify the notation.
By the inductive assumption, the head of $A_\mu L(\lambda')$ is the direct sum  of several copies of
$L(\tilde{\lambda}')$. By the definition of $\tilde{f}_i$, we have an epimorphism
$F_i^{(n)}L(\lambda')\twoheadrightarrow L(\lambda)$.
Since the functors $A_\mu$ and $F_{i}^{(n)}$ commute, \cite[Proposition 5.15]{SV}, we get an epimorphism
$F_i^{(n)}A_\mu L(\lambda')\twoheadrightarrow A_\mu L(\lambda)$. By \cite[Lemma 5.11]{CR}, any simple
$L(\hat{\lambda})$ appearing in the head of $F_i^{(n)}A_\mu L(\lambda')$ is not killed by
$E_i^{(n)}$. By the construction, $E_i L(\lambda')=0$.
Since $E_i$ also commutes with $A_\mu$, we see that $E_i A_\mu L(\lambda')=0$. It follows that
$E_i^{(n)}F_i^{(n)}A_\mu L(\lambda')$ is a multiple of $A_{\mu} L(\lambda')$. But $E_i^{(n)}L(\hat{\lambda})$
appears in the head of $A_{\mu} L(\lambda')$. It follows that $\tilde{e}_i^n \hat{\lambda}=\tilde{\lambda}'$
and hence $\hat{\lambda}=\tilde{\lambda}$. This completes the proof.
\end{proof}

Using Proposition \ref{Prop:cryst_commut_einfty}, we can now define an $\slf_\infty$-crystal structure commuting with the
$\hat{\slf}_e$-crystal on the whole set $\Part_\ell$. Namely, we declare that for each $\lambda$ with $p_{c}(\lambda)=0$
a crystal operator $C_f$ for $\hat{\slf}_e$ with $C_f \lambda\neq 0$ is an $\slf_{\infty}$-crystal embedding
of the component of $\lambda$ into $\Part_\ell$. Thanks to Proposition \ref{Prop:cryst_commut_einfty},
this indeed gives rise to an $\slf_\infty$-crystal on $\Part_\ell$. Since this crystal arises
from the Heisenberg categorical action, we call it the {\it Heisenberg crystal}.


\begin{Lem}\label{Lem:supp_Heis_cryst}
The number $q_c(\lambda)$ coincides with the depth of $\lambda$ in the $\slf_\infty$-crystal.
\end{Lem}
\begin{proof}
Note that $q_c(\tilde{f}_i\lambda')=q_c(\lambda')$ for any $\lambda'$. This follows from Lemma
\ref{Lem:Ind_support}. Using this and the construction of the $\slf_\infty$-crystal,
we reduce the proof to the case when $p_c(\lambda)=0$.  Here our claim follows from
\ref{SSS_Heis_cryst_p0}.
\end{proof}

\subsection{Computation of the Heisenberg crystal in the asymptotic chambers}\label{SS_comput_Heis}
Here we  are going to compute the $\slf_\infty$-crystal operators
under the condition that one of $s_j$ is much less than the others.

\begin{Prop}\label{Prop:cryst_part}
Suppose that $\kappa<0$ and $s_j<s_i-N$ for all $i\neq j$ and some $N>0$. If $\lambda$ is a multipartition
with $p_{c}(\lambda)=0$ and $|\lambda|\leqslant N$, then $\lambda^{(j)}$ is divisible by $e$ and $q_{c}(\lambda)=|\lambda^{(j)}|/e$.
Further, if $|\lambda|+e|\mu|\leqslant N$, then $(\tilde{a}_\mu\lambda)^{(i)}=\lambda^{(i)}$
for $i\neq j$ and $(\tilde{a}_\mu \lambda)^{(j)}=\lambda^{(j)}+e\mu$.
\end{Prop}
Note that this proposition implies Proposition \ref{Prop:domin_cryst}.
\begin{proof}
It follows from  \cite[Theorem 5.1]{cryst} that, under the assumptions of the proposition, $\lambda^{(j)}$ is divisible
by $e$. Indeed, under our assumption on $s_1,\ldots,s_\ell$, the $z$-signature of $\lambda^{(j)}$ will appear in
the end of the $z$-signature of $\lambda$, for all $z$. It follows that the reduced signatures of
$\lambda^{(j)}$ consist only of $+$'s. It is easy to see that this condition is equivalent to
$\lambda^{(j)}$ being divisible by $e$.

Now let us prove that $q_{c}(\lambda)\geqslant |\mu|$, where $e\mu=\lambda^{(j)}$. Set $\underline{\lambda}=(\lambda^{(1)},\ldots, \lambda^{(j-1)}, \varnothing, \lambda^{(j+1)},\ldots,\lambda^{(\ell)})$.
First, let us notice that $\Delta(\lambda)$ is the smallest (in the highest weight order) standard appearing in the filtration of $\operatorname{Ind}\Delta(\underline{\lambda})\boxtimes \Delta^A(e\mu)$ (here we have the induction
from $G(|\lambda|-e|\mu|,1,\ell)\times \Sym_{e|\mu|}$ to $G(|\lambda|,1,n)$). It follows that $L(\lambda)$
is in the head of $\operatorname{Ind}\Delta(\underline{\lambda})\boxtimes \Delta^A(e\mu)$ and hence in the head
of some object induced from a simple in the category $\mathcal{O}_c(G(|\lambda|-e|\mu|,1,\ell)\times \Sym_{e|\mu|})$
(that occurs in the composition series of $\Delta(\underline{\lambda})\boxtimes \Delta^A(e\mu)$).
By Lemma \ref{Lem:Ind_support}, $q_c(\lambda)\geqslant |\mu|$.

On the other hand the number of $\lambda$ with $p_{c}(\lambda)=0, q_{c}(\lambda)=|\mu|$
coincides with the number of $\lambda$ with $p_{c}(\lambda)=0, |\lambda^{(j)}|=e|\mu|$.
This follows from the fact that (\ref{eq:bijection}) is a bijection. Therefore $q_c(\lambda)=|\mu|$.

It remains to show that $\tilde{a}_\mu \underline{\lambda}=\lambda$. Let $\underline{\lambda}$
be minimal (with $p_c(\underline{\lambda})=q_c(\underline{\lambda})=0$ and hence
$\underline{\lambda}^{(j)}=\varnothing$) such that this fails.
Again, $L(\lambda)$ appears in the head of $\operatorname{Ind}\Delta(\underline{\lambda})\boxtimes \Delta^A(e\mu)$. Let $K$ denote the kernel of $\Delta^A(e\mu)
\twoheadrightarrow L^A(e\mu)$. Then $K$ does not contain any minimally supported object from $\mathcal{O}_\kappa^A(e|\mu|)$
in the head. Indeed, the category of minimally supported objects is semisimple (see \cite[Theorem 1.8]{Wilcox})
and the head of $\Delta^A(e\mu)$ is $L^A(e\mu)$.
Lemma \ref{Lem:Ind_support} implies that $L(\lambda)$ does not appear in the head of $\operatorname{Ind}\Delta(\underline{\lambda})\boxtimes
K$. So $L(\lambda)$ lies in the head of $\operatorname{Ind}\Delta(\underline{\lambda})\boxtimes L^A(e\mu)$.
Let us show that the only subquotient $L(\underline{\lambda}')$ of $\Delta(\underline{\lambda})$
such that $L(\lambda)$ lies in the head of $\operatorname{Ind}L(\underline{\lambda}')\boxtimes L^A(e\mu)$
is the top quotient $L(\underline{\lambda})$.
Indeed, if $L(\lambda)$ lies in the head of $\operatorname{Ind} L(\underline{\lambda}')\boxtimes L^A(e\mu)$,
then $\tilde{a}_\mu \underline{\lambda}'=\lambda$. It follows from Lemma \ref{Lem:Ind_support}
that $p_c(\underline{\lambda}')=q_c(\underline{\lambda}')=0$. As  $\underline{\lambda'}<\underline{\lambda}$,
we get a contradiction with the inductive assumption in the beginning of this paragraph.
\end{proof}

\begin{Rem}\label{Rem:Prop_generaliz}
The result of the previous proposition can be generalized to $p_c(\lambda)\neq 0$.
Here we divide $\lambda^{(j)}$ by $e$ with remainder: $\lambda^{(j)}=e\lambda'+\lambda''$.
Then $q_c(\lambda)=|\lambda'|$. This is easily from Proposition \ref{Prop:cryst_part}
combined with the fact that  the $\hat{\slf}_e$ and the $\slf_\infty$-crystals
commute.
\end{Rem}

\subsection{Wall-crossing bijections and crystal operators}\label{SS_wc_vs_cryst}
Now let us explain an interplay between wall-crossing bijections and crystal operators.
We will show that wall-crossing bijections through essential walls (defined
in \ref{SSS_essent_walls}) intertwine
the crystal operators for both $\g_c$- and $\slf_\infty$-crystals
(the latter is considered when all numbers $e\kappa s_1,\ldots, e\kappa s_\ell$
are integral).

First, let us list the essential walls.

\begin{Lem}\label{Lem:essent_walls}
The following list gives a complete collection of essential walls
for the group $G(\ell,1,n)$.
\begin{enumerate}
\item $\kappa=0$ for the parameters $c$, where the $\kappa$-component is a rational
number with denominator between $2$ and $n$.
\item $h_i-h_j=\kappa m$ with $i\neq j$ and $|m|<n$ -- for the parameters
$c$ satisfying $$h_i-h_j-\kappa m\in \Z+\frac{j-i}{\ell} \,(\Leftrightarrow s_i-s_j-m\in \kappa^{-1}\Z).$$
\end{enumerate}
\end{Lem}
\begin{proof}
The category $\OCat_c(n)$ is semisimple if and only if all blocks in
$\Part_\ell(n)$ consist of one element. The description of blocks is
provided in Lemma \ref{Lem:blocks}. So if $c$ is not of the form
described in (1) or (2), then $\OCat_c$ is semisimple. Conversely, for
$c$ as described in (1) and (2), we can find two multipartitions
$\lambda,\lambda'$ of the form $\lambda=\underline{\lambda}\sqcup b,
\lambda'=\underline{\lambda}\sqcup b'$, where $b,b'$ are unequal equivalent
boxes. The corresponding wall $\Pi$ is of the form $\Pi_{\lambda,\lambda'}$.
It is essential.

To complete the proof  we need to determine the parameters $c$ for which
the walls we found are essential for. Recall that, by the definition of
an essential wall, if $\Pi$ is essential for $c$, then, for a Weil generic
element $\hat{c}\in c+\Pi$, the category $\OCat_{\hat{c}}(n)$ is not semisimple.
Together with the description in the previous paragraph, this gives the
description of the parameters $c$ in (1) and (2).
\end{proof}

So, according to Remark \ref{Rem:walls_separation}, to understand which essential
walls separate two parameters $c,c'$, we need to compare the signs of the functionals
$\kappa, h_i-h_j-\kappa m$ (so that the corresponding walls are essential for $c,c'$)
on $c,c'$.

Below we write $\WC_{c-\psi\leftarrow c}$ for the sum of the wall-crossing functors
over all $n$. The summand corresponding to $n$ will be denoted by $\WC_{c-\psi\leftarrow c}(n)$.
The main result here is as follows.

\begin{Prop}\label{Prop:WC_cryst_commut}
The following is true.
\begin{enumerate}
\item Wall-crossing bijections $\wc_{c-\psi\leftarrow c}$ through essential walls commute with the crystal
operators $\tilde{e}_z,\tilde{f}_z$ for $\g_c$.
\item Consider the wall in (2) of Lemma \ref{Lem:essent_walls}.
 Then the corresponding wall-crossing bijections $\wc_{c-\psi\leftarrow c}$ commute with the crystal operators for $\slf_\infty$.
\item Consider the wall in (1).
Then $\wc_{c-\psi\leftarrow c}\circ \tilde{a}_\mu=\tilde{a}_{\mu^t}\circ \wc_{c-\psi\leftarrow c}$.
\end{enumerate}
\end{Prop}
\begin{proof}

From Proposition \ref{Prop:WC_restr_essent} we deduce that
\begin{equation}\label{eq:WC_res_commute}\underline{\WC}_{c-\psi\leftarrow c}\circ \Res^W_{\underline{W}}\cong
\Res^W_{\underline{W}}\circ \WC_{c-\psi\leftarrow c}.\end{equation}

{\it Proof of (1)}. Here we take $\underline{W}:=G(\ell,1,n-1)$.
Let us show, first, that \begin{equation}\label{eq:WC_comm_Ei}\WC_{c-\psi\leftarrow c}\circ E_z
\cong E_z\circ \WC_{c-\psi\leftarrow c}.\end{equation}

Since $\WC_{c-\psi\leftarrow c}$ is an equivalence $D^b(\mathcal{O}_c)\rightarrow
D^b(\mathcal{O}_{c-\psi})$, it maps blocks into blocks. Note that  any block
of $D^b(\mathcal{O}_?)$ is the derived category of a block of $\mathcal{O}_?$ for
$?=c, c-\psi$.

The block decompositions of $\OCat_c, \OCat_{c-\psi}$
on the level of $K_0$ are the same because $\psi$ is integral.
So the functor  $\WC_{c-\psi\leftarrow c}$
preserves the labels of blocks because  it acts as the identity on
the $K_0$-groups, see Proposition \ref{Prop:WC_K0}.
(\ref{eq:WC_comm_Ei}) follows now from
the construction of the functors $E_z$ in \ref{SSS_cat_KM}.

Now let us show that $\wc_{c-\psi\leftarrow c}\circ \tilde{e}_i\cong \tilde{e}_i\circ
\wc_{c-\psi\leftarrow c}$. Since $\wc_{c-\psi\leftarrow c}$
is a bijection, the claim that it intertwines  the crystal operators $\tilde{f}_i$ will follow.

Recall from \ref{SSS_WC_properties} that $\WC_{c-\psi\leftarrow c}$ is a perverse equivalence, let $\Cat^c_i
\subset \mathcal{O}_c, \Cat^{c-\psi}_{i}\subset \mathcal{O}_{c-\psi}$ denote the corresponding
filtration subcategories for $i=1,2$. (\ref{eq:WC_comm_Ei}) implies that the functors $E_z$ preserve those filtrations and
induce categorical $\g_c$-actions on the quotients. Since the functor $\WC_{c-\psi\leftarrow c}$ induces the
abelian (up to  homological shifts) equivalences of the  subquotient categories, we deduce that $\wc_{c-\psi\leftarrow c}\circ \tilde{e}_i\cong
\tilde{e}_i\circ \wc_{c-\psi\leftarrow c}$.

{\it Proof of (2)}. Here $\underline{W}=G(\ell,1,n-ek)\times \Sym_{ek}$.
So $\OCat_c(\underline{W})=\OCat_c(n-ek)\boxtimes \OCat^A_\kappa(em)$.
 The functor
$\underline{\WC}_{c-\psi\leftarrow c}: D^b(\OCat_c\boxtimes \OCat^A_\kappa(em))
\xrightarrow{\sim} D^b(\OCat_{c-\psi}\boxtimes \OCat^A_\kappa(em))$ decomposes as
$\WC_{c-\psi\leftarrow c}\boxtimes \operatorname{id}$.
So it follows from (\ref{eq:WC_res_commute}) that
\begin{equation}\label{eq:WC_commut_A}
\WC_{c-\psi\leftarrow c}\circ A_\mu\cong A_\mu\circ \WC_{c-\psi\leftarrow c}.
\end{equation}

Using part (1) we reduce the proof of (2) to $\wc_{c-\psi\leftarrow c}\circ \tilde{a}_\mu(\lambda)=
\tilde{a}_\mu\circ \wc_{c-\psi\leftarrow c}(\lambda)$ for $\lambda$ satisfying $p_c(\lambda)=q_c(\lambda)=0$.
Set $\lambda':=\wc_{c-\psi\leftarrow c}(\lambda), \tilde{\lambda}:=\tilde{a}_\mu(\lambda),
\tilde{\lambda}':=\wc_{c-\psi\leftarrow c}(\tilde{\lambda}), \bar{\lambda}':=\tilde{a}_\mu(\lambda')$.
We need to show that $\tilde{\lambda}'=\bar{\lambda}'$.

It follows from (\ref{eq:WC_commut_A}) that $A_\mu(\Cat^c_i)\subset
\Cat^{c}_i$. On the other hand, by Corollary \ref{Cor:wc_bij_supp},
we have $p_{c-\psi}(\lambda')=q_{c-\psi}(\lambda')=0$
and  $p_{c-\psi}(\tilde{\lambda}')=0, q_{c-\psi}(\tilde{\lambda}')=|\mu|$. Let $i$ be such that
$\lambda\in \Irr(\Cat^c_i/\Cat^c_{i+1})$. Consider the object
$M:=A_\mu\circ \WC_{c-\psi\leftarrow c}(L_c(\lambda))\in D^b(\OCat_{c-\psi})$.
By (P2) from \ref{SSS_WC_properties}, we have $H_j(M)=0$ for $j<i, H_j(M)\in \Cat^{c-\psi}_{i+1}$ for
$j>i$. Moreover, $L_{c-\psi}(\bar{\lambda}')$ is a unique simple subquotient
of $H_i(M)$ that satisfies $q_c(\bar{\lambda}')=|\mu|$ and $L_{c-\psi}(\bar{\lambda}')\not\in
\Cat^{c-\psi}_{i+1}$. Now let us observe that $M=\WC_{c-\psi\leftarrow c}\circ A_\mu (L_c(\lambda))$. From (P3) it follows that $L_{c-\psi}(\tilde{\lambda'})$
is the only simple subquotient of $H_i(M)$ satisfying $q_c(\tilde{\lambda}')=|\mu|$
and $L_{c-\psi}(\tilde{\lambda}')\not\in \Cat^{c-\psi}_{i+1}$, provided such a subquotient
exists at all. So we see that $\bar{\lambda}'=\tilde{\lambda}'$. This finishes the proof of (2).


{\it Proof of (3)}. Let $\WC^A_{+\leftarrow -}$ denote the wall-crossing functor for type A
categories $\mathcal{O}$ (going from $\kappa$ negative to $\kappa$ positive).
By (\ref{eq:WC_res_commute}), we have $\underline{\WC}_{c-\psi\leftarrow c}=
\WC_{c-\psi\leftarrow c}\boxtimes \WC^A_{+\leftarrow -}$, where the meaning of
$\underline{\WC}_{c-\psi\leftarrow c}$ is as in the beginning of the proof of (2).

Note that $\WC^A_{+\leftarrow -}:D^b(\OCat_{-}^A(ek))\xrightarrow{\sim}D^b(\OCat_+^A(ek))$
induces an abelian equivalence with a shift of categories
of modules with minimal support (that are subs in the corresponding perverse filtration).
We have a self-bijection $\mu\mapsto \mu'$ of $\mathcal{P}_1(k)$
such that $\WC^A_{+\leftarrow -}(L_-(e\mu))=L_+((e\mu')^t)[-k(e-1)]$. Similarly to the proof
of part (2), part (3) will follow if
we check that $\mu'=\mu^t$.

First, consider the case when $k=|\mu|=2$ so that we have two options for $\mu$:
either $(2)$ or $(1^2)$.  By \cite[Theorem 1.8]{EGL}, we have
\begin{equation}\label{eq:K0-}
[L^A_-(2e)]-[L^A_-(e^2)]=\sum_{i=0}^{2e-1}(-1)^i[\Delta^A_-(2e-i, 1^i)].
\end{equation}
Similarly, we have
\begin{equation}\label{eq:K0+}
[L^A_+(1^{2e})]-[L^A_+(2^e)]=-\sum_{i=0}^{2e-1}(-1)^i[\Delta^A_+(2e-i, 1^i)].
\end{equation}
By Proposition \ref{Prop:WC_K0}, $\WC_{+\leftarrow -}$ induces the identity map
$K_0(\OCat^A_-)\rightarrow K_0(\OCat^A_+)$. So it maps $[L^A_-(2e)]-[L^A_-(e^2)]$
to $[L^A_+(2^e)]-[L^A_+(1^{2e})]$. It follows that $\WC_{+\leftarrow -}L^A_-(2e)=L^A_+(2^e)[-2(e-1)],
\WC_{+\leftarrow -}L^A_-(e^2)=L^A_+(1^{2e})[-2(e-1)]$. So, indeed, in this case,
$\mu'=\mu^t$.

Now we are going to prove that $\mu'=\mu^t$ by induction on $|\mu|$ with the base $|\mu|=2$.
Note that \begin{equation}\label{eq:Ind_Sym}\Ind^{\Sym_{ek}}_{\Sym_{e(k-1)}\times \Sym_e}L^A_-(e\underline{\mu})\boxtimes L^A_-(e)=\bigoplus_\mu L^A_-(e\mu),\end{equation}
where the sum is taken over all $\mu$ obtained from $\underline{\mu}$ by adding a box.
To see this one can use the equivalence of $\OCat^A_-(ek)$ with $\mathcal{S}_q(ek)\operatorname{-mod}$,
where $q=\exp(\pi\sqrt{-1}/e)$ and $\mathcal{S}_q(ek)$ is the $q$-Schur algebra of degree $ek$.
By \cite[Proposition 3.1]{SV}, the equivalences $\OCat^A_-(ek)\xrightarrow{\sim} \mathcal{S}_q(ek)\operatorname{-mod}$
intertwines the induction functors and the tensor product functors. The subcategory of minimally
supported modules in $\OCat^A_-(ek)$ corresponds to the essential image of the
Frobenius  $\mathcal{S}_1(k)\operatorname{-mod}\rightarrow\mathcal{S}_q(ek)\operatorname{-mod}$. The Frobenius pull-back
intertwines the tensor product functors as well. So the equivalence of the subcategory
of minimally supported modules in $\OCat^A_-(ek)$ and $\Sym_k\operatorname{-mod}$
given by $L^A_-(e\mu)\mapsto S_\mu$
intertwines the induction functors. (\ref{eq:Ind_Sym}) follows. We also have a direct analog
of (\ref{eq:Ind_Sym}) for the category $\OCat^A_+(ek)$.

It follows from (\ref{eq:WC_res_commute}), (\ref{eq:Ind_Sym})
and its $+\,$-analog  that $\mu\mapsto \mu'$ is an automorphism
of the branching graph for  $\Sym_1\subset \Sym_2\subset\ldots \subset \Sym_n\subset\ldots$.
We can uniquely recover $\mu$ from the collection of all diagrams obtained from $\mu$
by removing a box when $|\mu|>2$. Namely, if there is more than one diagram in this collection,
then for $\mu$ we take the union  of these diagrams. Otherwise, $\mu$ is a rectangle
and also can be recovered uniquely.
This serves as the induction step in our proof of $\mu'=\mu^t$. The proof of (3) is now complete.
\end{proof}

In particular, this proposition allows to compute the wall-crossing bijection for
the type A categories $\mathcal{O}$. Take a partition $\lambda$ and divide it
by $e$ with remainder $\lambda=e\lambda'+\lambda''$. The partition $\lambda''$
is $e$-co-restricted meaning that a column of each height appears less than $e$ times.
Let $\Part^e$ denote the set of all $e$-co-restricted partitions. This is the connected
component of $\varnothing$ in the $\hat{\slf}_e$-crystal. This set has a
remarkable involution called the Mullineux involution, $\mathsf{M}$: it is defined as the only
map that send $\varnothing$ to itself and is twisted equivariant with respect to the
crystal operators: $\mathsf{M}(\tilde{f}_i\lambda)=\tilde{f}_{-i}\mathsf{M}(\lambda)$.
Note that the special case of $e=+\infty$ of this fact has been used in the end of
the proof of Proposition \ref{Prop:WC_cryst_commut}.

The next corollary follows from (1) and (3) of Proposition \ref{Prop:WC_cryst_commut}.

\begin{Cor}\label{Cor:type_A_wc}
The bijection $\wc^A_{+\leftarrow -}$ sends $\lambda=e\lambda'+\lambda''$
to $(e(\lambda'^t)+\mathsf{M}(\lambda''))^t$.
\end{Cor}

\subsection{Computation of wall-crossing bijections}\label{SS_wc_bij_comput}
Here we will use Propositions \ref{Prop:wc_bij_comp} and \ref{Prop:WC_cryst_commut}
to compute the wall-crossing bijections through the walls described in
(2) of Lemma \ref{Lem:essent_walls}. More precisely, using Proposition \ref{Prop:wc_bij_comp}
we reduce the computation to the case when $c$ is Weil generic in
$\param'$. Then we use (1) of Proposition \ref{Prop:WC_cryst_commut} to do the computation.

For $m\in \Z$, we define the self-bijection $\wc_m$ of $\Part_2$. For a box $b=(x,y)$
in the first diagram we set $\cont(b):=x-y$, for $b'=(x,y)$ in the second diagram
we set $\cont(b')=x-y+m$. We can
produce two crystal structures on $\Part_2$ using the cancellation recipe for addable/removable
boxes similar to the one used in \ref{SSS_crystal} (both our crystals will be special
cases of the crystals considered there). The crystal operators  $\tilde{e}^{[j]}_i$ (resp., $\tilde{f}^{[j]}_i$)
with $i\in \Z, j=1,2,$ for  both crystals will remove (resp., add) boxes with shifted content
$i$. What is different for the two structures is the order in which the boxes are listed.
For the crystal with $j=1$, we first list the addable/removable $i$-box from
the first diagram and then the box from the second diagram. For the crystal with
$j=2$, we do vice versa. Cancellation of addable/removable boxes (we cancel $-+$)
and the definition of the crystal operators is the same as in \ref{SSS_crystal}.

\begin{Lem}\label{Lem:cryst_comput_expl}
There is a unique bijection $\wc_m:\Part_2\rightarrow \Part_2$ that preserves
the total number of boxes and intertwines the two sets of crystal operators:
$\wc_m(\tilde{e}^{[1]}_i\lambda)=\tilde{e}_i^{[2]}\wc_m(\lambda)$ and $\wc_m(\tilde{f}^{[1]}_i\lambda)=
\tilde{f}^{[2]}_i\wc_m(\lambda)$ for all $i\in\Z$.
\end{Lem}
\begin{proof}
Let us describe the bipartitions $\lambda$ annihilated by all $\tilde{e}^{[1]}_i, i\in \Z$.
For all $i$, the signatures must look like $\varnothing, +,++$ or $-+$. So $\lambda^{(2)}=\varnothing$
and $\lambda^{(1)}$ can have only one removable box, that box must have content $m$.
Hence $\lambda^{(1)}$ is a rectangle with opposite vertices being the box $(1,1)$ and a box with content $m$.
The similar description works for the operators $\tilde{e}_i^{[2]}$. In that case,
$\lambda^{(1)}=\varnothing$ and $\lambda^{(2)}$ is a rectangle with vertex on
the diagonal with non-shifted content $-m$. In particular, we see that for
each $n$, there is not  more than one singular 2-partition of $n$ for either of the crystals.
Since $\wc_m$ has to map singular bi-partitions to singular ones, we see that the
requirement that $\wc_m$ preserves the number of boxes determines $\wc_m$ on the singular
bi-partitions uniquely. Since $\wc_m$ intertwines the crystal operators, its extension
to all 2-partitions is unique as well.
\end{proof}

Now we are ready to describe the wall-crossing bijection through the wall
 $h_i-h_j=\kappa m$.

\begin{Prop}\label{Prop:short_wc_comput}
Let $\Pi$ be the wall defined by $h_i-h_j-\kappa m=0$, let $c$ be a parameter
in a chamber adjacent to $\Pi$ and let $c-\psi$ lie in the chamber opposite
to that of $c$ with respect to $\Pi$.  Suppose that the linear function $h_i-h_j-\kappa m$
is positive on $\psi$.
Then  $\lambda':=\wc_{c-\psi\leftarrow c}(\lambda)$ is computed as follows:
$\lambda'^{(k)}=\lambda^{(k)}$ if $k\neq i,j$ and $(\lambda'^{(j)},\lambda'^{(i)})=\wc_m(\lambda^{(j)},\lambda^{(i)})$.
\end{Prop}
\begin{proof}
Thanks to Proposition \ref{Prop:wc_bij_comp}, we reduce the proof to the case when
$c$ is Weil generic on $c+\Pi$. In this case, if $b,b'$ are equivalent boxes,
we have $b$ in the $i$th diagram and $b'$ in the $j$th diagram, or vice versa
or $b,b'$ lie in the same diagram and have the same content.
We can add the same number to the $s_1,\ldots,s_\ell$ and assume that $s_j=0,
s_i=m$, while the other $\ell-2$ numbers are generic. We will have $\ell-2$
collections of $\slf_\infty$-crystal operators, each collection acts
on its own partition. We will have another $\slf_\infty$-crystal acting
on partitions $i,j$ in one of two ways described above. The claim of this
proposition follows now from Proposition \ref{Prop:WC_cryst_commut} combined with the  uniqueness part of Lemma \ref{Lem:cryst_comput_expl}.
\end{proof}

\begin{Rem}
Here we explain how to compute the wall-crossing bijection through the wall $\kappa=0$.
The case of $\ell=1$ is done in Corollary \ref{Cor:type_A_wc}. The general case
reduces to that one as follows. In the computation of the bijection we can assume
that $s_1,\ldots,s_\ell$ are Weil generic. In this case the direct sum of  induction functors
from $S_{n_1}\times\ldots \times S_{n_k}$ to $G(\ell,1,n_1+\ldots+n_k)$ is a labelling preserving
equivalence $\OCat_{\kappa}^{\boxtimes m}\xrightarrow{\sim} \OCat_c$.
So the wall-crossing bijection $\mathcal{P}_\ell\rightarrow \mathcal{P}_\ell$
is the product of $\ell$ copies of the wall-crossing bijections
$\mathcal{P}_1\rightarrow \mathcal{P}_1$.
\end{Rem}

\subsection{Summary}\label{SS_comput_summ}
Let us summarize the computation of $p_c(\lambda)$ and $q_c(\lambda)$.

First, the number $p_c(\lambda)$ is the depth of $\lambda$ in the
$\g_c$-crystal, see \ref{SSS_crystal}. If $\kappa$ is irrational,
then $q_c(\lambda)=0$, and we are done.

So suppose from now on that $\kappa$ is rational, let $e$ be the denominator.
We may assume that $\kappa<0$, otherwise we switch $(\kappa,s_1,\ldots,s_\ell)$
to $(-\kappa,-s_1,\ldots,-s_\ell)$ and $\lambda$ to $\lambda^t$.

We may assume all numbers $\kappa e s_1,\ldots, \kappa e s_\ell$ are integral,
we can reduce to this case using Corollary \ref{Cor:decomp_supp}. Also we can assume that
$p_c(\lambda)=0$, we can reduce to this case by replacing $\lambda$
with the singular element in the connected component containing $\lambda$ in the $\hat{\slf}_e$-crystal.

Now let $|\lambda|=n$. If there is $j$ such that $s_j<s_i-n$ for
all $i\neq j$, then $q_c(\lambda)=|\lambda^{(j)}|/e$. In general,
we can reduce to the case when there is such $j$ by crossing
walls of the form $h_a-h_{b}=\kappa m$. Each time we cross the wall
we pick a parameter $(\kappa,s_1,\ldots,s_\ell)$ in the neighboring
chamber and modify $\lambda$ by applying the wall-crossing bijection
from Proposition \ref{Prop:short_wc_comput}.

\subsection{Chambers, combinatorially}
Let us fix a parameter $c=(\kappa,s_1,\ldots,s_\ell)$ with $\kappa=-\frac{r}{e}$,
where $r>0,e>1$, $r,e$ are coprime, and $rs_1,\ldots,rs_\ell\in \Z$. Recall that
the collection $(s_1,\ldots,s_\ell)$ is defined up to adding the same scalar to all
$s_i$'s. With our conventions, this scalar must lie in $\frac{1}{r}\Z$.
We are going to describe the walls of
a ``combinatorial" chamber containing $c$. The description is going to be independent
of $n$, so, when we fix $n$, some essential chambers are going to be unions of several combinatorial
chambers.

Let us start by defining combinatorial
chambers. These will be specified by linear orders of boxes with the  same residue.
Recall that we record a multipartition as a collection of shapes consisting of
unit square boxes in $\ell$ coordinate planes.

A combinatorial chamber will be parameterized by an $\ell$-tuple of pairs of integers.
Pick a residue $\alpha$ mod $e$. Since $r$ and $e$ are coprime, it makes
sense to speak about $\alpha$-boxes. The choice of our parameter $c$ gives a linear
ordering on $\alpha$-boxes in all possible partitions, see \ref{SSS_blocks_order}. Pick an interval of length $\ell$ (it contains exactly one box
from each of the $\ell$ coordinate planes) and record (in a decreasing way)
the pairs $(m_j,i_j)$, $j=1,\ldots,k,$ where $m_j$ is the (unshifted)
content of the $j$th box and $i_j$ is the number of the plane where this
box is located.

Depending on our choice of $\alpha$ and of the first  box in the interval,
we will get different sequences that will be called equivalent. The equivalence
relation is generated by
\begin{align*}&((m_1,i_1),\ldots, (m_\ell,i_\ell))\sim ((m_1+m,i_1),\ldots, (m_\ell+m,i_\ell)),\\
&((m_1,i_1),\ldots, (m_\ell,i_\ell))\sim ((m_2,i_2),\ldots, (m_\ell,i_\ell),(m_1-e,i_1)),
\end{align*}
where $m$ is an arbitrary integer.
We declare parameters $(\kappa,s_1,\ldots,s_\ell), (\kappa',s_1',\ldots,s_\ell')$
with integral difference (i.e., $\kappa'-\kappa\in \Z, r(s_i-s_j)-r(s_i'-s_j')\in e\Z$) to lie in the same combinatorial chamber if the $\ell$-sequences
of pairs produced from these parameters are equivalent. Note that for parameters
lying in the same combinatorial chamber, the orders on the categories $\mathcal{O}$
are the same for all $n$.

Note that the affine Weyl group $\hat{S}_\ell$ acts on the set of chambers  via
\begin{align*}&\sigma_j.((m_1,i_1),\ldots, (m_\ell,i_\ell))=((m_1,i_1),\ldots, (m_{j+1},i_{j+1}),(m_j,i_j),\ldots, (m_\ell,i_\ell)), 1<j<\ell,\\
&\sigma_\ell.((m_1,i_1),\ldots, (m_\ell,i_\ell))=((m_\ell+e,i_\ell),(m_2,i_2),\ldots, (m_1-e,i_1)).\end{align*}
Here $\sigma_1,\ldots,\sigma_\ell$ denote the simple reflections in $\hat{S}_\ell$.

We say that two combinatorial chambers are {\it adjacent} if the corresponding
$\ell$-tuples of pairs are obtained from one another either by applying some $\sigma_j$.
In other words, neighboring chambers correspond to a minimal
perturbation of the order on boxes. Therefore, for each fixed combinatorial chamber
there is $n_0\in \Z_{>0}$ such that this chamber is an essential chamber for
all $n>n_0$.

Let us explain how to determine walls between adjacent combinatorial chambers. This will allow us to determine
which wall-crossing bijection we need to apply to get between these two chambers. Suppose that
we have permuted pairs $(m_j,i_j),(m_{j+1},i_{j+1})$. Then the corresponding wall
is $h_{i_j}-h_{i_{j+1}}=\kappa(m_j-m_{j+1})$. So to get from the initial chamber
to its neighbor we will need to apply the bijection $\mathfrak{wc}_m$ with
$m=m_{j+1}-m_j$ to the partitions $i_j,i_{j+1}$.

Let us finish this section with an example. Let $\kappa=-\frac{1}{3}, s_1=0,s_2=1,s_3=2$.
Recall that different choices of $\alpha$ give equivalent sequences.
Take $\alpha$ to be $2$ modulo $3$.
Then we get the following triple of pairs: $((0,3),(1,2),(2,1))$. The other two choices of $\alpha$
give equivalent triples. 
The three neighbor
combinatorial chambers correspond to the following triples:
\begin{itemize}
\item[(i)] $((1,2),(0,3),(2,1))$.
\item[(ii)]  $((0,3),(2,1),(1,2))$.
\item[(iii)] $((5,1),(1,2),(-3,3))$.
\end{itemize}
For a parameter in chamber (i), we can take, for example, $\kappa'=-\frac{4}{3}, s_1'=0, s_2=1, s_3=\frac{1}{2}$.
To get to this chamber, we need to apply $\mathfrak{wc}_1$ to the diagrams number $3$ and $2$.



\subsection{Example of computation}\label{SS_comput_example}
Here we will compute the numbers $p_c(\lambda),q_c(\lambda)$
for $\ell=2, |\lambda|\leqslant 3, \kappa<0, e=2$, even $s_1$ and $s_2$.

\subsubsection{Chambers}
The  essential walls are $h_1-h_2=m\kappa$ for $m=-2,0,2$ and $\kappa=0$.
So we have four chambers with $\kappa<0$. Here are the pairs of integers
corresponding to these combinatorial chambers (we only take chambers
adjacent to these walls).

\begin{enumerate}
\item $(0,1),(2,2)$,
\item $(0,1),(0,2)$ (equivalent to $(2,2),(0,1)$),
\item $(0,1), (-2,2)$ (equivalent to $(0,2),(0,1)$).
\item $(0,1),(-4,2)$ (equivalent to $(-2,2),(0,1)$).
\end{enumerate}

The chambers (1) and (4) are asymptotic.

\subsubsection{Wall-crossing bijections}
Let us start with the bijection between (1) and (2) that is $\mathfrak{wc}_{-2}$. It is the identity on $\mathcal{P}_2(1)$ and $\mathcal{P}_2(2)$.

The bijection sends the singular bipartition $(1^3,\varnothing)$ in chamber (1) to the singular
bipartition $(\varnothing, 3)$ in chamber (2). Further, it sends $(1^2,1)$ to $(1^3,\varnothing)$,
$(\varnothing,3)$ to $(1,2)$, and $(1,2)$ to $(1^2,1)$. Finally, the bijection $\mathfrak{wc}_{-2}$
fixes all other bi-partitions.
%
The bijection $\mathfrak{wc}_0$ from chamber (2) to chamber (3) just swaps the components
of the partition.

The bijection $\mathfrak{wc}_2$ from chamber (3) to chamber (4) sends $(3,\varnothing)$
to $(\varnothing,1^3), (\varnothing,1^3)$ to $(1,1^2)$, $(2,1)$ to $(3,\varnothing)$
and $(1,1^2)$ to $(2,1)$.

%
%
\subsubsection{Supports}
{\it Chamber (1)}. The $\hat{\slf}_2$-crystal looks as follows:

$\tilde{f}_0(\varnothing, \varnothing)=(\varnothing,1), \tilde{f}_1(\varnothing,\varnothing)=0$.

$\tilde{f}_0(\varnothing,1)=(1,1), \tilde{f}_1(\varnothing,1)=(\varnothing,1^2)$.

$\tilde{f}_0(1,\varnothing)=0, \tilde{f}_1(1,\varnothing)=(1^2,\varnothing)$.

$\tilde{f}_0(\varnothing, 1^2)=(\varnothing,1^3), \tilde{f}_1(\varnothing, 1^2)=(\varnothing,(21))$.

$\tilde{f}_0(\varnothing, 2)=(\varnothing, 3), \tilde{f}_1(\varnothing,2)=(\varnothing,21)$.

$\tilde{f}_0(1,1)=0, \tilde{f}_1(1,1)=(1,1^2)$.

$\tilde{f}_0(1^2,\varnothing)=(1^2,1), \tilde{f}_1(1^2,\varnothing)=(21,\varnothing)$.

$\tilde{f}_0(2,\varnothing)=(2,1), \tilde{f}_1(2,\varnothing)=0$.

The following bipartitions have $p(\lambda)=|\lambda|$: $(\varnothing, \varnothing),
(\varnothing,1), (\varnothing,1^2),(1,1) (\varnothing,1^3), (\varnothing, (21))$,
and $(1,1^2)$.

The following bipartitions have $p(\lambda)=|\lambda|-1$: $(1,\varnothing),(1^2,\varnothing),
(1^2,1),(21,\varnothing)$.

The following bipartitions have $p(\lambda)=|\lambda|-2$: $(\varnothing,2),(2,\varnothing), (\varnothing,3),
 (2,1)$.

The following bipartitions of 3 have $p(\lambda)=0$: $(3,\varnothing),(1^3,\varnothing),(1,2)$.

The following bipartitions have $q(\lambda)=1$: $(\varnothing,2),(1,2),(\varnothing,3)$.
All other bipartitions have $q(\lambda)=0$.

In particular, the above computation implies that the bi-partitions $(3,\varnothing),(1^3,\varnothing)$
corresponding to finite dimensional simples lie in different pieces of the filtration
that makes the wall-crossing functor from chamber (1) to chamber (2) perverse.

{\it Chamber (2)}. The pairs $(p(\lambda),q(\lambda))$ are the same as in chamber (1)
except the following four cases.
\begin{itemize}
\item $p(1^3,\varnothing)=2, q(1^3,\varnothing)=0$,
\item $p(\varnothing,3)=q(\varnothing,3)=0$,
\item $p(1,2)=q(1,2)=1$,
\item $p(1^2,1)=0, q(1^2,1)=1$.
\end{itemize}

{\it Chamber (3)}. Obtained from chamber (2) by swapping the components
of a bipartition.

{\it Chamber (4)}. Obtained from chamber (1) by swapping the components
of a bipartition.
\section{Appendix}\label{S_App}
\subsection{Case $\kappa=0$}\label{S_App_kappa0}
Here we will explain how to compute the supports of
the irreducible modules in $\OCat_c(n)$ in the case when $\kappa=0$.
In this case $H_c(n)=H_c(1)^{\otimes n}\#\Sym_n$ and so
an object in the category $\mathcal{O}_c(n)$
is the same things as an $\Sym_n$-equivariant object in $\mathcal{O}_c(1)^{\otimes n}$.
Recall that in this case $p_c(\lambda)=0$, by convention.
The number $q_c(\lambda)$ is computed as follows.
The simple modules in $\OCat_c(1)$ are labelled by
the numbers from $1$ to $\ell$. Let $I$ denote the
subset of $\{1,\ldots,\ell\}$ consisting of the indexes
$i$ such that the corresponding module has dimension
of support equal to $1$. Then $q_c(\lambda)=\sum_{i\in I}|\lambda^{(i)}|$.

\subsection{Groups $G(\ell,r,n)$}\label{S_gen_groups}
Let $\ell,n$ be the same as before and let $r$ divide $\ell$.
Then we can consider the normal subgroup $G(\ell,r,n)$ consisting of all
elements of the form $\sigma\eta$, where $\sigma\in \Sym_n$
is an arbitrary element and $\eta=(\eta_{(1)},\ldots,\eta_{(n)})\in (\Z/\ell\Z)^n$
satisfies $\prod_{i=1}^n \eta_{(i)}^r=1$. This is a complex reflection
group (in its action on $\h=\C^n$). In particular, for $r=\ell=2$
we get the Weyl group of type $D_n$.

The following lemma is elementary.

\begin{Lem}\label{Lem:conj_classes}
Suppose that $n>2$. Then  every conjugacy class in
$G(\ell,1,n)$ contained in $G(\ell,r,n)$ is a single conjugacy class
in the latter.
So we have $\ell/r$ conjugacy classes of reflections in
$G(\ell,r,n)$.
\end{Lem}
Note that the claim of the lemma is false when $n=2$: the class
of a transposition from $\Sym_n$ in $G(\ell,1,2)$ is contained
in $G(\ell,r,2)$ and splits into the union of several conjugacy classes
there.

Let $c\in \param$ be such that the values of $c$ on the conjugacy classes
not intersecting $G(\ell,r,n)$ are zero. Define a parameter $\underline{c}$ for $G(\ell,r,n)$
as the restriction of $c$ to $G(\ell,r,n)\cap S$. Let $H_c$ and $H_{\underline{c}}$ be the Cherednik algebras
for $G(\ell,1,n)$ and $G(\ell,r,n)$, respectively. Note that the group
$G(\ell,1,n)$ acts on $H_{\underline{c}}$ by automorphisms. Then we have
$H_c=H_{\underline{c}}\#_{G(\ell,r,n)}G(\ell,1,n)$. It follows that
$\OCat_c(n)$ is the category of $G(\ell,1,n)$-equivariant objects
in $\OCat_{\underline{c}}(G(\ell,r,n))$. Recall that by a $G(\ell,1,n)$-equivariant
object in  $\OCat_{\underline{c}}(G(\ell,r,n))$ we mean $M\in
\OCat_{\underline{c}}(G(\ell,r,n))$ that is also a $G(\ell,1,n)$-module with the following two compatibility conditions
\begin{itemize}
\item the actions of $G(\ell,r,n)\subset G(\ell,1,n),H_{\underline{c}}$
agree,
\item and $M$ is a $G(\ell,1,n)$-equivariant $H_{\underline{c}}$-module.
\end{itemize}
This  reduces
questions about characters/supports from
$\OCat_{\underline{c}}(G(\ell,r,n))$ to $\OCat_c(n)$.

\subsection{Three commuting crystals}\label{S_commut_cryst}
Assume that $\kappa=-\frac{1}{e}, s_1,\ldots,s_\ell\in \Z$.
We write ${\bf s}$ for the $\ell$-tuple $(s_1,\ldots,s_\ell)$ and
$|{\bf s}|$ for $s_1+\ldots+s_\ell$. We assume, for simplicity,
that $|{\bf s}|=0$. We write $\OCat_{\bf s}$ for $\OCat_{c}$.

We have established two commuting crystals on $\Part_\ell$. They are crystal
analogs of the two commuting actions on the Fock space $\mathcal{F}_{\bf s}$:
of $\hat{\slf}_e$ and of the Heisenberg algebra. Recall that one way to realize
the Fock space is via the level-rank duality. Namely, $\sum_{{\bf s}, |{\bf s}|=0}\mathcal{F}_{\bf s}$
is a module over $\hat{\slf}_e\times \mathfrak{heis}\times \hat{\slf}_\ell$,
and $\mathcal{F}_{\bf s}$ is a weight space for $\hat{\slf}_\ell$, see
\cite[Section 2.1]{Uglov} for the quantum version of this construction. The name
``level-rank duality'' is explained by the fact that the representation
in $\sum_{{\bf s}, |{\bf s}|=s}\mathcal{F}_{\bf s}$ has level $\ell$
for the algebra $\hat{\slf}_e$ and level $e$ for the algebra $\hat{\slf}_\ell$.
Also we have $\sum_{{\bf s}, |{\bf s}|=0}\mathcal{F}_{\bf s}=\sum_{{\bf s'},|{\bf s'}|=0 }
\mathcal{F}'_{\bf s'}$, where ${\bf s'}=(s_1',\ldots,s'_e)$ and $\mathcal{F}'_{\bf s'}$
denotes the level $e$ Fock space with multi-charge ${\bf s'}$ for $\hat{\slf}_\ell$.

We get two commuting
crystals for $\hat{\slf}_e$ and $\hat{\slf}_\ell$ (usually realized via
abaci). It should be possible to check that the Heisenberg crystal commutes
with the $\hat{\slf}_\ell$-crystal combinatorially\footnote{After the first version of this paper
appeared this has been done in
\cite{Gerber}.}. What we would like
to do, however, is to explain the categorical meaning of these three
crystals that should easily imply the commutativity.

It is known that the level-rank duality is categorified by the Koszul duality,
\cite{RSVV,Webster_new}. Namely, the category  $\bigoplus_{{\bf s}, |{\bf s}|=0}\OCat_{\bf s}$
is standard Koszul, and its Koszul dual category is $\bigoplus_{{\bf s'}, |{\bf s'}|=0}\OCat'_{\bf s'}$,
where $\OCat'_{\bf s'}$ stands for the category $\mathcal{O}$ for the groups $G(e,1,?)$
and the parameters $\kappa':=-\frac{1}{\ell}, {\bf s'}$. Below we are going to sketch
a new approach to the Koszul duality that nicely incorporates all three categorical actions
(two Kac-Moody actions and one Heisenberg action).

Our approach is based on  the work of Bezrukavnikov and Yun, \cite{BY}.
They consider geometric versions of (singlular, parabolic) affine
categories $\mathcal{O}$ for Kac-Moody Lie algebras, in particular,
for $\hat{\slf}_m$. For two compositions ${\bf s},{\bf s'}$ of $m$
with $\ell$ and $e$ parts, respectively, we have the parabolic-singular
category $\mathcal{O}^{aff}_{{\bf s},{\bf s}'}$, where ${\bf s}$ encodes
the ``parabolicity'' and ${\bf s}'$ encodes the singularity. We have commuting functors
$E_i, i=1,\ldots,e,$ and $E_j', j=1,\ldots,\ell$, between the categories
$\mathcal{O}^{aff}$ that change ${\bf s}'$ and ${\bf s}$, respectively,
and their biadjoints $F_i, F_j'$.
Further, we have endo-functors of   $\mathcal{O}^{aff}_{{\bf s},{\bf s}'}$,
Gaitsgory's central functors, that commute with the $E_i$'s, $F_i$'s and $E'_j$'s,
$F'_j$'s.

The Koszul (or more precisely, Ringel-Koszul) duality is between the
categories $\mathcal{O}^{aff}_{{\bf s},{\bf s}'}$ and
$\mathcal{O}^{aff}_{{\bf s}',{\bf s}}$. It switches the functors
$E_i$'s and $E'_j$'s and preserves Gaitsgory's central functors.

Now we need to relate the categories $\mathcal{O}^{aff}_{{\bf s},{\bf s'}}$
with the categories $\mathcal{O}_{{\bf s}}(n)$. This should be done as in
\cite{VV_proof,RSVV}. Namely, one can pick $m$ large enough and consider
a ``polynomial truncation'' of  $\mathcal{O}^{aff}_{{\bf s},{\bf s'}}$.
The Ringel-Koszul duality restricts to the polynomial truncations.
The functors
$E_i,F_i$ on $\mathcal{O}^{aff}_{{\bf s},{\bf s'}}$
will become the Kac-Moody  categorification functors,
while Gaitsgory's central functors will give rise to the categorical
Heisenberg action.

\end{document}